\documentclass[11pt,leqno]{article}
\usepackage{amsmath, amscd, amsthm, amssymb, graphics, xypic, mathrsfs, setspace, fancyhdr, times, enumerate}
\entrymodifiers={+!!<0pt,\fontdimen22\textfont2>}
\DeclareFontFamily{OT1}{pzc}{}
\DeclareFontShape{OT1}{pzc}{m}{it}%
             {<-> s * [1.195] pzcmi7t}{}
\DeclareMathAlphabet{\mathscr}{OT1}{pzc}%
                                 {m}{it}
\usepackage[colorlinks=true,pagebackref=true]{hyperref} 
\hypersetup{backref}

\newcommand{\tensor}{\otimes}
\newcommand{\colim}{\operatorname{colim}}
\newcommand{\Spec}{\operatorname{Spec}}

\newcommand{\isomto}{{\stackrel{\sim}{\;\longrightarrow\;}}}
\newcommand{\isomt}{{\stackrel{{\scriptscriptstyle{\sim}}}{\;\rightarrow\;}}}

\newcommand{\sma}{{\scriptstyle{\wedge}}}

\renewcommand{\O}{{\mathcal O}}
\renewcommand{\hom}{\operatorname{Hom}}

\newcommand{\Z}{{\mathbb Z}}
\newcommand{\N}{{\mathbb N}}

\newcommand{\aone}{{\mathbb A}^1}
\newcommand{\pone}{{\mathbb P}^1}

\newcommand{\gm}[1]{{{\mathbf G}_{m}^{#1}}}
\renewcommand{\L}{{\mathcal L}}

\newcommand{\hop}[1]{\mathscr{H}_{\bullet}({#1})}

\newcommand{\bpi}{\boldsymbol{\pi}}

\newcommand{\Nis}{\operatorname{Nis}}

\newcommand{\Sm}{\mathscr{Sm}}

\newcommand{\Spc}{\mathscr{Spc}}

\newcommand{\Ab}{\mathscr{Ab}}

\newcommand{\K}{{{\mathbf K}}}

\newcommand{\hsnis}{\mathscr{H}_s^{\Nis}(k)}
\newcommand{\hspnis}{\mathscr{H}_{s,\bullet}^{\Nis}(k)}

\newcommand{\Addresses}{{
 \bigskip
 \footnotesize

 A.~Asok, Department of Mathematics, University of Southern California, 3620 S. Vermont Ave.,
  Los Angeles, CA 90089-2532, United States; \textit{E-mail address:} \url{asok@usc.edu}

  \medskip

 J.~Fasel, Institut Fourier - UMR 5582, Universit\'e Grenoble Alpes, 100, rue des math\'ematiques, F-38402 Saint Martin d'H\`eres; France \textit{E-mail address:} \url{jean.fasel@gmail.com}
}}

\newcounter{intro}
\theoremstyle{plain}
\newtheorem{thm}{Theorem}[subsection]

\newtheorem{lem}[thm]{Lemma}
\newtheorem{cor}[thm]{Corollary}
\newtheorem{prop}[thm]{Proposition}
\newtheorem*{claim*}{Claim}  

\newtheorem*{thm*}{Theorem}
\newtheorem*{problem*}{Problem}

\newtheorem{thmintro}{Theorem}

\theoremstyle{definition}
\newtheorem{defn}[thm]{Definition}
\newtheorem{construction}[thm]{Construction}
\newtheorem{notation}[thm]{Notation}

\theoremstyle{remark}
\newtheorem{rem}[thm]{Remark}

\newtheorem{ex}[thm]{Example}

\numberwithin{equation}{section}

\begin{document}
\pagestyle{fancy}
\renewcommand{\sectionmark}[1]{\markright{\thesection\ #1}}
\fancyhead{}
\fancyhead[LO,R]{\bfseries\footnotesize\thepage}
\fancyhead[LE]{\bfseries\footnotesize\rightmark}
\fancyhead[RO]{\bfseries\footnotesize\rightmark}
\chead[]{}
\cfoot[]{}
\setlength{\headheight}{1cm}

\author{Aravind Asok\thanks{Aravind Asok was partially supported by National Science Foundation Awards DMS-0966589 and DMS-1254892.}\and Jean Fasel\thanks{Jean Fasel was partially supported by the DFG Grant SFB Transregio 45.}}

\title{{\bf Comparing Euler classes}}
\date{}
\maketitle

\begin{abstract}
We establish the equality of two definitions of an Euler class in algebraic geometry: the first definition is as a ``characteristic class" with values in Chow-Witt theory, while the second definition is as an ``obstruction class."  Along the way, we refine Morel's relative Hurewicz theorem in ${\mathbb A}^1$-homotopy theory, and show how to define (twisted) Chow-Witt groups for geometric classifying spaces.
\end{abstract}

\begin{footnotesize}
\tableofcontents
\end{footnotesize}
\newpage
\section{Introduction}
Suppose $k$ is a perfect field having characteristic unequal to $2$, $X$ is a $d$-dimensional smooth $k$-scheme, and $\xi: \mathcal{E} \to X$ is a rank $r$ vector bundle on $X$.  A choice of isomorphism $\theta: \det \xi \isomt \O_X$ is called an {\em orientation} of $\xi$.  There are (at least) two ways to define an ``Euler class" $e(\xi)$ of $\xi$ that provides an  obstruction to existence of a nowhere vanishing section of $\xi$; both definitions are simpler in case $\xi$ is oriented.  The goal of this note is to prove equivalence of these definitions, which we now recall.

Using the notation of \cite{FaselChowWitt}, one possible definition is as follows (this is the ``characteristic class" approach mentioned in the abstract).  If $s_0: X \to \mathcal{E}$ is the zero section of $\mathcal{E}$, then there are pullback and Gysin pushforward homomorphisms in Chow-Witt groups.  There is a canonical element $\langle 1 \rangle \in \widetilde{CH}^0(X)$, and we define
\[
e_{cw}(\xi) := (\xi^*)^{-1}(s_0)_* \langle 1 \rangle \in \widetilde{CH}^r(X,\operatorname{det}(\xi)^{\vee})
\]
(see \cite[Definition 13.2.1]{FaselChowWitt}; the definition there is equivalent to forming the product with this class using the ring structure in Chow-Witt groups of \cite{FaselChowWittRing} by the excess intersection formula of \cite{FaselExcessIntersection}).  This definition of Euler class is functorial for pullbacks and, modulo unwinding the definition of the Gysin pushforward, coincides with the definition of Euler class given in \cite[\S 2.1]{BargeMorel} for oriented vector bundles (\cite[Proposition 13.4.1]{FaselChowWitt}).

An alternative definition comes from \cite[Remark 8.15]{MField}, where one constructs an Euler class as {\em the} primary obstruction to existence of a non-vanishing section of $\xi$.  In that case, if $Gr_r$ denotes the infinite Grassmannian, and $\gamma_r$ is the universal rank $r$ vector bundle on $Gr_r$, the first non-trivial stage of the Moore-Postnikov factorization in $\aone$-homotopy theory of the map $Gr_{r-1} \to Gr_r$ gives rise to a canonical morphism $Gr_r \to K^{\gm{}}(\K^{MW}_r,r)$ (see \cite[Appendix B]{MField} for a discussion of twisted Eilenberg-Mac Lane spaces in our setting) yielding a canonical (equivariant) cohomology class
\[
o_r \in H^r_{\Nis}(Gr_r,\K^{MW}_r(\det \gamma_r^{\vee})).
\]
Given any smooth scheme $X$ and an $\aone$-homotopy class of maps $\xi: X \to Gr_r$ ``classifying" a vector bundle $\xi$ as above, the pullback of $o_r$ along $\xi$ yields a class
\[
e_{ob}(\xi) := \xi^*(o_r).
\]
This definition of Euler class is evidently functorial for pullbacks as well (note: with this definition $k$ is not required to have characteristic unequal to $2$).

In \cite{BargeMorel}, and subsequently in \cite{FaselChowWitt}, the (twisted) Chow-Witt groups are defined as follows.  If $\mathbf{I}^r$ denotes the unramified sheaf associated with the $r$-th power of the fundamental ideal in the Witt ring, there is a canonical morphism of sheaves $\K^M_r/2 \to \mathbf{I}^r/\mathbf{I}^{r+1}$ \cite[p. 78]{MMilnor}.  Using this morphism, define a sheaf by taking the fiber product of $\mathbf{I}^r$ and $\K^M_r$ over $\mathbf{I}^r/\mathbf{I}^{r+1}$ and the Chow-Witt groups as the $r$-th cohomology of the resulting fiber product sheaf.  By the universal property of the fiber product, there is a canonical morphism from $\K^{MW}_r$ to the fiber product sheaf, and this morphism induces an isomorphism $H^n(X,\K^{MW}_r) \isomt \widetilde{CH}^r(X)$ (see, e.g., \cite[Theorem 5.47]{MField}).  In fact, by the Milnor conjecture on quadratic forms, now a theorem \cite{Orlov07}, the morphism from $\K^{MW}_r$ to the fiber product sheaf is an isomorphism.  More generally, we will observe that these isomorphism can be suitably ``twisted by a local system" to yield an isomorphism
\begin{equation}\label{eqn:canonicalisom}
H^r(X,\K^{MW}_r(\det \xi^{\vee})) \longrightarrow \widetilde{CH}^r(X,\det \xi^{\vee})
\end{equation}
that will be used to identify these two groups; this isomorphism is decribed in detail in Theorem \ref{thm:identification1} but requires Proposition \ref{prop:gmaction}, which shows that the ``action" defining the local system housing the obstruction-theoretic Euler class coincides with that in the sheaf-theoretic definition of the twisted Chow-Witt groups.

Under the isomorphism of the previous paragraph, it makes sense to compare the two Euler classes described above.  In \cite[Remark 8.15]{MField}, Morel asserts that the two definitions of Euler class given above coincide, but provides no proof.  The main result of this paper provides justification for Morel's assertion, and can be viewed as an analog in algebraic geometry of \cite[Theorem 12.5]{MilnorStasheff}.  We will call a vector bundle $\xi: \mathcal{E} \to X$ on a smooth $k$-scheme $X$ {\em oriented} if there is a specified trivialization of the determinant.

\begin{thmintro}[See Theorem \ref{thm:comparison}]
\label{thmintro:comparison}
If $k$ is a perfect field having characteristic unequal to $2$, and if $\xi: \mathcal{E} \to X$ is an oriented rank $r$ vector bundle on a smooth $k$-scheme $X$, then, under the identifications described above,
\[
e_{ob}(\xi) = u e_{cw}(\xi),
\]
where $u \in GW(k)^{\times}$ is a unit.
\end{thmintro}

The method of proof we propose is classical and is likely the one envisaged by Morel: we establish this result by the method of the universal example.  Nevertheless, we felt it useful to provide a complete proof of the above result for at least three other reasons.  First, as is perhaps evident from the length of this note, a fair amount of effort is required to develop the technology necessary for comparing the various constructions appearing in the definitions of Euler classes; furthermore, some of the results established here will be used elsewhere (e.g., the Blakers-Massey theorem is used in \cite{WickelgrenWilliams}).  Second, as Morel observes in \cite[Remark 8.15]{MField}, in combination with his $\aone$-homotopy classification of vector bundles and the theory of the Euler class \cite[Theorems 8.1 and 8.14]{MField}, the above result completes the verification of the main conjecture of \cite{BargeMorel}.  Third, the results of this paper are already used in \cite[Lemma 3.3]{AsokFaselSpheres} and therefore implicitly in \cite{AsokFaselThreefolds} and \cite{AsokFaselA3minus0}.  Finally, we observe that the Chow-Witt Euler class is much more straightforward to explicitly compute than the obstruction theoretic Euler class.


\subsubsection*{Overview of contents}
Section \ref{s:chowwitt} is devoted to establishing functoriality properties and various different models of Chow-Witt groups, together with some technical results about Chow-Witt groups for which we could not find a good reference.  In particular, this section includes a construction of the canonical isomorphism mentioned in the introduction. Section \ref{s:eulerthom} is devoted to extending the definition of Chow-Witt groups to classifying spaces (including the infinite Grassmannian), and for reviewing various things related to the obstruction theoretic Euler class.  Section \ref{s:relativehurewicz}, establishes a Blakers-Massey theorem in $\aone$-homotopy theory, which we believe, as remarked above, has independent value.  The results of Section \ref{s:relativehurewicz} are independent of the first two sections, but are integral to identifying the obstruction theoretic Euler class as a ``transgressive class"; this notion, motivated by the classical Serre spectral sequence, is explained in more detail in Section \ref{s:transgression}.  Finally, Section \ref{s:transgression} also contains the proof of the main Theorem, and uses all the preceding results.  More detailed descriptions of the contents and results are provided at the beginning of each section.

\subsubsection*{Acknowledgements}
The first author would like to thank Jeff Giansiracusa for discussions about the relationship between the transgression of the fundamental class and the Euler class in classical algebraic topology.  We also thank the referee for a number of suggestions to improve the clarity of the presentation.

\subsubsection*{Notation and preliminaries}
Our notation regarding $\aone$-homotopy theory, fiber sequences, etc., follows the conventions we laid out in \cite[\S 2]{AsokFaselSpheres}.  We have made a significant effort to keep this work as self-contained as possible, but as is probably clear from the technical nature of the theorem statement, this was largely impossible.  When discussing $\aone$-homotopy theory, we have used techniques and ideas from \cite{MV} and \cite{MField} rather freely (hopefully with sufficiently precise references that the interested reader can follow).  When discussing Chow-Witt groups, we have used the results of \cite{FaselChowWitt} rather freely.  Since we expect this paper will be read in conjunction with one of \cite{AsokFaselSpheres, AsokFaselThreefolds} or \cite{AsokFaselA3minus0} we have attempted to keep terminology consistent with those papers.

Any time we refer to the results of \cite{MField}, the reader should understand that $k$ is assumed perfect.  For the most part, we follow the notation of \cite{FaselChowWitt} for Chow-Witt theory.  Any time we refer to the results of Chow-Witt theory, the reader should understand that $k$ has characteristic unequal to $2$.   Thus, for simplicity, the reader can assume throughout that $k$ is perfect and has characteristic unequal to $2$ everywhere.  (The especially wary reader may also want to assume that $k$ is infinite for technical reasons as explained in \cite[\S 1]{AsokFaselSpheres}).

\section{Chow-Witt groups revisited}
\label{s:chowwitt}
In this section, we revisit the definition of Chow-Witt groups.  Subsection \ref{ss:chowwittfunctoriality} is devoted to establishing functoriality for pullbacks of arbitrary morphisms of smooth schemes, which is only implicit in previous work.  Then, we pass to the sheaf-theoretic approach to Chow-Witt groups.  Subsection \ref{ss:contractionsandactions} is devoted to studying twists of strictly $\aone$-invariant sheaves, a language that is necessary for comparison with Chow-Witt groups twisted by a line bundle.  Subsection \ref{ss:twistedchowwittgroups} establishes the main comparison result (i.e., Theorem \ref{thm:identification1}).  In particular, we observe that a sheaf-theoretic definition of Chow-Witt groups studied by Morel can be {\em functorially} identified with the version studied in \cite{FaselChowWitt}.  Finally Subsection \ref{ss:furtherproperties} establishes some further technical results about Chow-Witt groups that will be useful in the remainder of the paper.

\subsection{Functoriality of (twisted) Chow-Witt groups revisited}
\label{ss:chowwittfunctoriality}
In \cite[Definition 7.1]{FaselChowWittRing}, the second author gave a definition of a pullback in (twisted) Chow-Witt for an arbitrary morphism $f: X \to Y$ of smooth schemes; changing the notation slightly we will write $f^*$ for this pullback (rather than $f^!$).  More precisely, if $f$ is as above, and $\mathcal{L}$ is a line bundle on $Y$, then there is a pullback morphism
\[
f^*: \widetilde{CH}^i(Y,\mathcal{L}) \longrightarrow \widetilde{CH}^i(X,f^*{\mathcal L}).
\]
Recall that we factor $f$ as the composition of the graph map $\Gamma_f: X \to X \times Y$, which is a regular embedding, and the projection $p_Y: X \times Y \to Y$, which is a smooth morphism and the morphism $f^*$ is defined as the composite $\Gamma_f^* p_Y^*$ (we use \cite[Remark 5.6]{FaselChowWittRing} to define the morphism in the twisted case for a regular embedding).  Unfortunately, the functoriality properties of this construction are only implicit in \cite{FaselChowWittRing} so for the sake of completeness, we spell out the proofs here.

Recall from \cite[Proposition 7.4]{FaselChowWittRing} that in case $f$ is flat, the morphism $f^*$ described above coincides with the usual pullback for Chow-Witt groups as studied in \cite[Corollaire 10.4.3]{FaselChowWitt}.  In particular, the pullback on Chow-Witt groups is functorial for smooth morphisms; we use this observation repeatedly in the sequel.  In each statement below, the line bundles are suppressed from notation for convenience.

\begin{lem}
\label{lem:tworegularembeddings}
Given a diagram of the form
\[
X \stackrel{i}{\longrightarrow} Y \stackrel{p}{\longrightarrow} Z,
\]
where $i$ is a regular embedding of codimension $d$, $p$ is a smooth morphism of relative dimension $n$ and $pi$ is a regular embedding of codimension $d-n$, the equality $(pi)^* = i^*p^*$ holds.
\end{lem}

\begin{proof}
Form the fiber product diagram
\[
\xymatrix{
X \times_Z Y \ar[r]^-j\ar[d]_-{p'} & Y \ar[d]^-p \\
X \ar[r]_-{pi}\ar[ur]^i & Z.
}
\]
Since $p$ is smooth, it follows that $p'$ is smooth as well.  Moreover, the morphism $i$ determines a section $s: X \to X \times_Z Y$.  Since $p'$ is smooth by \cite[B.7.5]{Fulton} or \cite[Expose VIII 1.3]{SGA6} it follows that $s$ is a regular embedding.  Also, $j$ is a regular embedding.  Since $i = js$, it follows from \cite[Theorem 5.11]{FaselChowWittRing} that $i^* = s^*j^*$.  Therefore, $i^*p^* = s^*j^*p^*$.  Now, by \cite[Lemma 5.7]{FaselChowWittRing}, we know that $j^*p^* = (p')^*(pi)^*$, so we conclude that $s^*j^*p^* = s^*(p')^*(pi)^*$.  Now, since $s$ is a section, $p's = id_X$.  Applying \cite[Lemma 5.10]{FaselChowWittRing} we conclude that $id = (p's)^* = s^*(p')^*$, which concludes the proof.
\end{proof}

Next, we establish that we can use any factorization of a morphism $f: X \to Y$ of smooth schemes as the composite of a regular embedding followed by a smooth morphism to define the pullback.

\begin{lem}
\label{lem:independenceoffactorization}
If $f: X \to Y$ is a morphism of smooth schemes, and $X \stackrel{i_1}{\to} P_1 \stackrel{p_1}{\to} Y$ and $X \stackrel{i_2}{\to} P_2 \stackrel{p_2}{\to} Y$ are two factorizations of $f$ as the composition of a regular embedding followed by a smooth morphism, then $i_1^*p_1^* = i_2^*p_2^*$.
\end{lem}

\begin{proof}
Consider the fiber product $P_1 \times_Y P_2$ and let $p_1'$ and $p_2'$ be the induced maps from the fiber product to $P_1$ and $P_2$.  There is a diagonal map $X \stackrel{(i_1,i_2)}{\to} P_1 \times_Y P_2$, which is again a regular embedding.  Now, Lemma \ref{lem:tworegularembeddings} allows us to conclude that $(i_1,i_2)^*(p_1^\prime)^* = i_1^*$ and $(i_1,i_2)^*(p_2^\prime)^* = i_2^*$.  We conclude using the equality $(p_1')^*p_1^* = (p_2')^*p_2^*$ given by the usual functoriality for pullbacks with respect to smooth morphisms \cite[Corollaire 10.4.3]{FaselChowWitt}.
\end{proof}

Finally, we can establish functoriality of pullbacks; we will use this result repeatedly in the sequel without explicit mention.

\begin{thm}
If $f: X \to Y$ and $g: Y \to Z$ are morphisms of smooth schemes, then $(gf)^* = f^*g^*$.
\end{thm}

\begin{proof}
Contemplate the diagram
\[
\xymatrix{
X \ar[r]^-{\Gamma_f}\ar[dr]_-f & X \times Y \ar[r]^-{\Gamma_g'}\ar[d]^-{p_Y} & X \times Y \times Z \ar[d]^-{p_{Y,Z}} \\
 & Y \ar[dr]_-g\ar[r]^-{\Gamma_g} & Y \times Z \ar[d]^-{p_Z} \\
  & & Z
}
\]
where $\Gamma_g'$ is the pullback of $\Gamma_g$ along $p_{Y,Z}$.  All vertical morphisms are smooth and all horizontal morphisms are regular embeddings.  This diagram is commutative by construction.

The pullbacks in the only square appearing in the above diagram commute by \cite[Lemma 5.7]{FaselChowWittRing} since $p_{Y,Z}$ and $p_Y$ are both smooth and $\Gamma_g$ and $\Gamma_g'$ are regular embeddings.  Now, we just observe that by Lemma \ref{lem:independenceoffactorization} we can use the factorization of $gf: X \to Z$ as $\Gamma_g'\Gamma_f: X \to X \times Y \times Z$ and $p_Zp_{Y,Z}: X \times Y \times Z \to X$ to compute $(gf)^*$.
\end{proof}

\subsection{Contractions and actions}
\label{ss:contractionsandactions}
\begin{defn}
\label{defn:strictlyaoneinvariant}
Recall that a sheaf $\mathbf{A}$ of abelian groups on $\Sm_k$ is {\em strictly $\aone$-invariant} if its cohomology presheaves $U \mapsto H^i(U,\mathbf{A})$ are $\aone$-invariant, i.e., the maps $U \times \aone \to U$ induce bijections $H^i(U,\mathbf{A}) \to H^i(U \times \aone,\mathbf{A})$ for every integer $i \geq 0$ and every $U \in \Sm_k$.
\end{defn}

Suppose $U$ is a smooth $k$-scheme.  If $\mathbf{A}$ is a strictly $\aone$-invariant sheaf, the contraction ${\mathbf A}_{-1}$ is defined sectionwise using the cokernel of the pullback along the projection map $U \times \gm{} \to U$:
\[
{\mathbf A}_{-1}(U) := \operatorname{coker}(\mathbf{A}(U) \longrightarrow \mathbf{A}(U \times \gm{})).
\]
This cokernel is again a strictly $\aone$-invariant sheaf, and we inductively define $\mathbf{A}_{-n}$ by the formula $\mathbf{A}_{-n} := (\mathbf{A}_{-n+1})_{-1}$.  If $\mathbf{A} = \mathbf{B}_{-1}$ for some strictly $\aone$-invariant sheaf $\mathbf{B}$, we will say that $\mathbf{A}$ is a contraction.

If $\mathbf{A}$ is a strictly $\aone$-invariant sheaf, then we can define a morphism $\gm {} \times \mathbf{A}_{-1} \to \mathbf{A}$ as follows.  If $U$ is a smooth $k$-scheme, take an element $a \in \gm{}(U) = \O_U(U)^{\times}$ and view this element as a morphism $U \to \gm{}$.  The composite map
\[
U \stackrel{\Delta}{\longrightarrow} U \times U \stackrel{a \times id}{\longrightarrow} \gm{} \times U
\]
then induces a pullback map
\[
ev_{a}: \mathbf{A}(\gm{} \times U) \longrightarrow \mathbf{A}(U).
\]
The difference $ev_{a} - ev_1$ is, by construction, trivial on any element in the image of the pullback $\mathbf{A}(U) \to \mathbf{A}(U \times \gm{})$ and therefore induces a map
\[
a \cup : \mathbf{A}_{-1}(U) \longrightarrow \mathbf{A}(U).
\]
There is thus an induced map
\[
\gm{}(U) \times \mathbf{A}_{-1}(U) \longrightarrow \mathbf{A}(U),
\]
which is functorial in $U$ by construction.  The map just mentioned extends to a bilinear map
\[
\K^{MW}_1 \times \mathbf{A}_{-1} \longrightarrow \mathbf{A}
\]
by \cite[Lemma 3.48]{MField}.

If $E/k$ is a field extension and $a \in E^{\times}$ we can consider the expression $\langle a \rangle := 1 + \eta a \in \K^{MW}_0(E)$.  If $U$ is a smooth scheme, sending $a \in k(U)$ to $\langle a \rangle$ defines a homomorphism
\[
\gm{}(k(U)) \to \K^{MW}_0(k(U))^{\times}.
\]
Since the restriction maps $\gm{}(U) \to \gm{}(k(U))$ and $\K^{MW}_0(U) \to \K^{MW}_0(k(U))$ are injective (by construction in the latter case), there is an induced morphism of sheaves $\gm{} \to (\K^{MW}_0)^{\times}$.  This homomorphism extends to a morphism of sheaves of rings
\[
\Z[\gm{}] \longrightarrow \K^{MW}_0.
\]

If $\mathbf{A}$ is a strictly $\aone$-invariant sheaf, the sheaf $\mathbf{A}_{-1}$ admits an action by $\gm{}$ defined as follows.  If $a \in \O^\times_U(U)$, as above we view this as a map $a: U \to \gm{}$ and consider the composite map
\[
\xymatrix@C=4.7em{\gm{} \times U\ar[r]^-{id_{\gm{}}\times (a,Id_U)}  &  \gm{} \times \gm{} \times U \ar[r]^-{m \times id_U} &  \gm{} \times U},
\]
%
where $m$ is the multiplication morphism on $\gm{}$.  The resulting composite, which we shall denote $\tilde{a}$, is an isomorphism with inverse given using the the inverse of the unit $a$.  If we consider $\tilde{a}^*: \mathbf{A}(\gm{} \times U) \to \mathbf{A}(\gm{} \times U)$, then the assignment $a \mapsto \tilde{a}^*$ defines an action of $\gm{}(U)$ on $\mathbf{A}(\gm{} \times U)$.  If we consider the trivial action of $\gm{}(U)$ on $\mathbf{A}(U)$, then the map $\mathbf{A}(U) \to \mathbf{A}(\gm{} \times U)$ induced by pullback along the projection is $\gm{}(U)$-equivariant and there is thus an induced action map $\gm{}(U) \times \mathbf{A}_{-1}(U) \to \mathbf{A}_{-1}(U)$.  This construction is clearly functorial in $U$ and there is thus an action map
\[
\gm{} \times \mathbf{A}_{-1} \to \mathbf{A}_{-1}.
\]
By \cite[Lemma 3.49]{MField}, the action of $\gm{}$ on $\mathbf{A}_{-1}$ just described factors through an action of $\K^{MW}_0$ on $\mathbf{A}_{-1}$ by means of the morphism $\Z[\gm{}] \to \K^{MW}_0$-described in the previous paragraph.

\begin{rem}
\label{rem:kmwgmaction}
If $\mathbf{A} = \K^{MW}_n$, then ${\mathbf A} = (\K^{MW}_{n+1})_{-1}$ and the action map $\K^{MW}_0 \times \K^{MW}_n \to \K^{MW}_n$ coincides with the standard action of $\K^{MW}_0$ on $\K^{MW}_n$.
\end{rem}

\begin{defn}
\label{defn:twistedsheaves}
Suppose $X$ is a smooth $k$-scheme, and $\L$ is an invertible sheaf on $X$.  Write $L$ for the associated line bundle on $X$ and $L^{\circ}$ for the $\gm{}$-torsor on $X$ underlying $L$, i.e., the complement of the zero section in $L$.  If $\mathbf{A}$ is a strictly $\aone$-invariant sheaf carrying an action of $\gm{}$, set
\[
\mathbf{A}(\L) := \Z[L^{\circ}] \tensor_{\Z[\gm{}]} \mathbf{A}.
\]
\end{defn}

\begin{construction}[Pullbacks]
\label{construction:pullbacksfortwistedsheaves}
Suppose $f: X \to Y$ is a morphism of smooth schemes, and $\mathcal{L}$ is a line bundle on $Y$.  In that case, we can consider the line bundle $f^*{\mathcal L}$ on $X$.  If, as above, $L$ is the geometric vector bundle associated with ${\mathcal L}$, and $f^*L$ is the geometric vector bundle on $X$ associated with $f^*{\mathcal L}$, then we can identify $f^*L$ with the fiber product scheme $X \times_Y L$.  Moreover, there is an induced morphism $f^*L^{\circ} \to L^\circ$.  If $\mathbf{A}$ is as in Definition \ref{defn:strictlyaoneinvariant}, then there is an induced map $f^* {\mathbf A}({\mathcal L}) \to {\mathbf A}(f^*{\mathcal L})$.  As a consequence, there is, for any integer $i \geq 0$, an induced pullback map
\[
f^*: H^i(Y,{\mathbf A}({\mathcal L})) \longrightarrow H^i(X,\mathbf{A}(f^*{\mathcal L})).
\]
\end{construction}


\subsection{The sheaf-theoretic approach to twisted Chow-Witt groups}
\label{ss:twistedchowwittgroups}
If $\mathbf{A}$ is a contraction, then by \cite[Remarks 5.13-14]{MField}, the twisted sheaf $\mathbf{A}(\L)$, viewed as a sheaf on the small Zariski site of $X$, admits a (``twisted") Gersten resolution.  More precisely, if $F$ is a field, and if $\Lambda$ is a $1$-dimensional $F$-vector space, set $\mathbf{A}(F;\Lambda) := \mathbf{A}(F) \tensor_{\Z[F^{\times}]} \Z[\Lambda \setminus 0]$.  The sheaf $\mathbf{A}(\L)$ admits a flasque resolution by a complex whose underlying graded abelian group is of the form:
\[
C^*(X,\L,\mathbf{A}) := \bigoplus_{x \in X^{(n)}} \mathbf{A}_{-n}(\kappa_x; \Lambda^n_{\kappa_x} (\mathfrak{m}_x/\mathfrak{m}_x^2)^\vee \tensor \L_x);
\]
with differential defined on \cite[p. 122]{MField}.

\begin{prop}
\label{prop:gerstencomplexescoincide}
If $X$ is a smooth scheme, $\L$ is a line bundle on $X$, and $\mathbf{A} = \K^{MW}_r$ (with the $\gm{}$-action described in Remark \ref{rem:kmwgmaction}), then the complex $C^*(X,\L,\K^{MW}_r)$ coincides with the complex of \textup{\cite[D{\'e}finition 10.2.10]{FaselChowWitt}}.
\end{prop}

\begin{proof}
By the Milnor conjecture on quadratic forms \cite{Orlov07}, there are canonical identifications
\[
\K^{MW}_n \isomto \K^M_n \times_{\mathbf{I}^n/\mathbf{I}^{n+1}} \mathbf{I}^n.
\]

The fiber product complex of \cite[D{\'e}finition 10.2.10]{FaselChowWitt} (in the untwisted case) is precisely the Gersten resolution of the fiber product sheaf $\K^M_n \times_{\K^M_n/2} \mathbf{I}^n$ and the isomorphism of sheaves in the previous paragraph, together with compatibility of that isomorphism with contractions shows that this complex coincides with the Gersten resolution of $\K^{MW}_r$.  It remains to establish that this isomorphism is compatible with twists as well, but we leave this to the reader.
\end{proof}

\subsubsection*{Identification of pullbacks}
Suppose $Y$ is a smooth $k$-scheme, and $u\in \O_Y(Y)$ is a regular function with vanishing locus $D$, a smooth closed subscheme with open complement $U$. Write $\kappa:D\to Y$ for this closed embedding and $i:U\to Y$ for the corresponding open embedding. For any point $x\in U$ and any $m\in\Z$, there is a multiplication by $u$ homomorphism
\[
\mathbf{K}_m^{MW}(k(x))\to \mathbf{K}_{m+1}^{MW}(k(x))
\]
defined by $\alpha\mapsto [u]\cdot \alpha$.  By \cite[Proposition 3.17]{MField} this homomorphism commutes, up to multiplication by $\epsilon=-\langle -1\rangle$, with differentials in the Gersten complex on $U$.

\begin{rem}
\label{rem:suppressingtwists}
Analogously, multiplication by $u$ defines a corresponding homomorphism for Milnor-Witt $K$-theory twisted by local orientations (in the sense of Definition \ref{defn:twistedsheaves}).  For the sake of unburdening our already suffering notation, we have, in this section, suppressed all line bundle twists.  The arguments given below extend with only notational changes to the twisted setting.
\end{rem}

If $n\in\N$ and $x\in U^{(n)}$, we have a residue homomorphism
\[
d:\mathbf{K}_{m+1}^{MW}(k(x),\Lambda^n_{\kappa_x} (\mathfrak{m}_x/\mathfrak{m}_x^2)^\vee)\to \bigoplus_{x\in X^{(n+1)}}\mathbf{K}_{m}^{MW}(k(x),\Lambda^n_{\kappa_x} (\mathfrak{m}_x/\mathfrak{m}_x^2)^\vee)
\]
We can split the right-hand term as the direct sum over points in $X^{(n+1)}\cap D$ and the direct sum over points in $U^{(n+1)}$. Consequently, we get homomorphisms of abelian groups
\[
C^n(U,\K_m^{MW})\to C^{n+1}_D(X,\K_m^{MW})\oplus C^{n+1}(U,\K_m^{MW}).
\]
We write $\partial_D$ for the first component of this homomorphism and we observe that the second component is just the differential $d_U$ of the Gersten complex on $U$.

\begin{lem}
\label{lem:commutativity}
The diagram
\[
\xymatrix{C^n(U,\K_m^{MW})\ar[r]^-{\partial_D}\ar[d]_-{d_U} & C^{n+1}_D(X,\K_m^{MW})\ar[d]^-{d_D} \\
C^{n+1}(U,\K_m^{MW})\ar[r]_-{\partial_D} & C^{n+2}_D(X,\K_m^{MW})}
\]
anti-commutes.
\end{lem}

\begin{proof}
If $\alpha\in C^n(U,\K_m^{MW})$, then we may view it as an element of $C^n(X,\K_m^{MW})$.  The following equality holds:
\[
0=(d\circ d)(\alpha)=d(\partial_D(\alpha)+d_U(\alpha))=d\circ\partial_D(\alpha)+d\circ d_U(\alpha).
\]
Now $\partial_D(\alpha)$ is supported on $D$ and therefore $d\circ\partial_D(\alpha)=d_D\circ\partial_D(\alpha)$. On the other hand, we have $d\circ d_U(\alpha)=\partial_D\circ d_U(\alpha)+d_U\circ d_U(\alpha)=\partial_D\circ d_U(\alpha)$ since $d_U\circ d_U=0$. The claim follows.
\end{proof}

Gathering the above constructions, we obtain a diagram of the form:
\[
\xymatrix{C^n(U,\K_{m}^{MW})\ar[r]^-{[u]}\ar[d]_-{d_U} & C^n(U,\K_{m+1}^{MW})\ar[r]^-{\partial_D}\ar[d]_-{d_U} & C^{n+1}_D(X,\K_{m+1}^{MW})\ar[d]^-{d_D}\ar[r]^-{(\kappa_*)^{-1}} & C^n(D,\K_m^{MW})\ar[d]^-{d_D} \\
C^{n+1}(U,\K_{m}^{MW})\ar[r]_-{[u]} & C^{n+1}(U,\K_{m+1}^{MW})\ar[r]_-{\partial_D} & C^{n+2}_D(X,\K_{m+1}^{MW})\ar[r]_-{(\kappa_*)^{-1}} &  C^{n+1}(D,\K_m^{MW})}.
\]
The right-hand horizontal maps are the inverses of the push-forward isomorphisms $\kappa_*:C^n(D,\K_m^{MW})\to C^{n+1}_D(X,\K_{m+1}^{MW})$.  Regarding the squares in this diagram: the inner square anti-commutes by Lemma \ref{lem:commutativity}, the left square commutes up to multiplication by $- \langle -1 \rangle$ as mentioned before Remark \ref{rem:suppressingtwists} and the right hand square commutes.  Therefore, the outer square commutes up to multiplication by $\langle -1 \rangle$.

Definining $\delta_n$ to be the composite $\langle (-1)^n\rangle (\kappa_*)^{-1}\circ \partial_D \circ [u]$, we obtain a morphism of complexes
\[
\delta:C^*(U,\K_m^{MW})\to C^*(D,\K_m^{MW})
\]
that induces, upon taking cohomology, the homomorphism of the bottom line in the diagram above \cite[Definition 5.5]{FaselChowWittRing}.

The pull-back along closed immersions for Chow-Witt groups involves deformation to the normal cone, $\delta$, homotopy invariance and the pull-back along smooth morphisms.  To show that this pull-back coincides with the sheaf-theoretic pullback, it suffices to show that the morphism $\delta$ induces a morphism of flasque resolutions that reduces to the pull-back on the Nisnevich sheaf $\K_m^{MW}$ at the level of $H^0$; this follows by definition of the sheaf itself.  For convenient reference, we summarize these observations in the following result.

\begin{thm}
\label{thm:identification1}
The twisted Chow-Witt group $\widetilde{CH}^i(X,\L)$ as defined in \textup{\cite[D{\'e}finition 10.2.16]{FaselChowWitt}} coincides with $H^i(X,\K^{MW}_i(\L))$ via the identification of \textup{Proposition \ref{prop:gerstencomplexescoincide}}.  Under this identification, the pullback of \textup{Construction \ref{construction:pullbacksfortwistedsheaves}} coincides with the pullback for Chow-Witt groups (see \textup{\S \ref{ss:chowwittfunctoriality}}).
\end{thm}

\subsection{Further properties}
\label{ss:furtherproperties}
In this section, we establish two technical properties: a base-change result, which appears as Theorem \ref{thm:basechange} below and an excision result, which appears as Lemma \ref{lem:stabilization}.

\subsubsection*{Base-change}
\begin{thm}
\label{thm:basechange}
Suppose $f: X \to Y$ is a regular embedding of smooth schemes fitting into a Cartesian square of smooth schemes of the form
\[
\xymatrix{
X' \ar[r]^-{v}\ar[d]_-g & X \ar[d]^-f \\
Y' \ar[r]_-u & Y.
}
\]
Suppose that the diagram is transversal in the sense that the natural morphism $N_{X'}Y'\to v^*N_XY$ is an isomorphism. Then $u^* f_* ( - ) = g_*v^* ( - )$.
\end{thm}

\begin{proof}
As usual, we omit the potential line bundles in the formula for the sake of conciseness. We follow the arguments of \cite[Theorem 32]{FaselExcessIntersection}. Recall first from \cite[\S 2.6]{FaselExcessIntersection} that if $f:X\to Y$ is a regular embedding of smooth schemes, the deformation to the normal cone provides a commutative diagram
\[
\xymatrix{X\ar[r]^-{\delta_0}\ar[d]_-s & X\times\aone\ar[d] & X\ar[l]_-{\delta_1}\ar[d]^-f \\
N_XY\ar[r]_-{i_0} & D(X,Y) & Y\ar[l]^-{i_1}}
\]
whose squares are Cartesian. Here, $\delta_i$ is the inclusion $X\times\{i\}\to X\times \aone$, $s:X\to N_XY$ is the zero section of the normal bundle and $i_0$, $i_1$ are closed embeddings. Following the arguments of \cite[Lemma 28]{FaselExcessIntersection}, we see that this diagram provides an isomorphism
\[
d(X,Y):H^r_{X}(N_XY,\K_s^{MW})\to H^r_{X}(Y,\K_s^{MW}).
\]
The second step of the proof is to prove that this isomorphism fits into the commutative diagram
\begin{equation}\label{eqn:deformation}
\xymatrix{H^{r}(X,\K_s^{MW},\det N_XY)\ar[r]^-{s_*}\ar@{=}[d] & H^{r+d}_X(N_XY,\K_{s+d}^{MW})\ar[d]^-{d(X,Y)} \\
H^{r}(X,\K_s^{MW},\det N_XY)\ar[r]_-{f_*} & H^{r+d}_X(Y,\K_{s+d}^{MW}). }
\end{equation}
This reduces to \cite[Lemma 2.2]{FaselExcessIntersection}, whose analogue in our context is easily deduced using Theorem \ref{thm:identification1} (in the case of a principal Cartier divisor).

Consider now our starting diagram
\[
\xymatrix{
X' \ar[r]^-{v}\ar[d]_-g & X \ar[d]^-f \\
Y' \ar[r]_-u & Y.
}
\]
Using \cite[Corollaire 12.2.8]{FaselChowWitt} and the usual factorization of $u$ (and $v$), we may assume that $u$ and $v$ are regular embeddings. As the extension of support commutes with pull-backs, it is then sufficient to prove that the following diagram commutes
\[
\xymatrix{H^r(X^\prime,\K_s^{MW},\det N_{X^\prime}Y^\prime)\ar[d]_-{g_*} & H^r(X,\K_s^{MW},\det N_XY)\ar[d]^-{f_*}\ar[l]_-{v^*} \\
H^{r+d}_{X^\prime}(Y^\prime,\K_{s+d}^{MW}) &  H^{r+d}_{X}(Y,\K_{s+d}^{MW})\ar[l]^-{u^*}.  }
\]
Under our assumptions, we get a Cartesian square
\begin{equation}\label{eqn:easier}
\xymatrix{X^\prime\ar[r]^-v\ar[d]_-{s^\prime} & X\ar[d]^-s \\
N_{X^\prime}Y^\prime\ar[r]_-{v^\prime} &N_XY}
\end{equation}
where $s$ and $s^\prime$ are the zero sections and $v^\prime$ is a closed embedding. Using the fact that $d(X,Y)$ commutes with pull-backs, we see from Diagram (\ref{eqn:deformation}) that we are reduced to prove the result for Diagram (\ref{eqn:easier}). Now we can use the deformation to the normal cone to $X^\prime$ in $X$ to understand the pull-back associated to $v$. It follows from \cite[Lemma 5.8]{FaselChowWittRing} that the normal cone to $N_{X^\prime}Y^\prime$ in $N_XY$ is canonically isomorphic to $N_{X^\prime}X\oplus N_{X^\prime}Y^\prime$ and we can use the functoriality of the deformation to the normal cone to reduce to prove the formula for the diagram
\[
\xymatrix{X^\prime\ar[r]^-t\ar[d]_-{s^\prime} & N_{X^\prime}/X\ar[d]^-s \\
N_{X^\prime}Y^\prime\ar[r]_-{t^\prime} & N_{X^\prime}X\oplus N_{X^\prime}Y^\prime}
\]
where all morphisms are zero sections. This case follows from \cite[Corollaire 12.2.8]{FaselChowWitt} as the pull-backs along zero sections are inverse of pull-back along projections for vector bundles.
\end{proof}

\subsubsection*{An excision result}
\begin{lem}
\label{lem:stabilization}
If $X$ is an $n$-dimensional smooth scheme over a perfect field $k$, $\L$ is a line bundle on $X$ and $Y \subset X$ is a closed subscheme of codimension $c$ with open complement $j: U \to X$, then the induced map
\[
j^*: \widetilde{CH}^i(X,\L) \longrightarrow \widetilde{CH}^i(U,j^*\L)
\]
is bijective for $i < c-1$.
\end{lem}

\begin{proof}
The Chow-Witt groups of $X$ are defined using an explicit Gersten-type complex $C_j(X,G,\mathcal{L})$ \cite[D{\'e}finition 10.2.13]{FaselChowWitt}. If $U$ is an open subscheme, then we have a surjective map of complexes $C_j(X,G,\mathcal{L})\to C_j(U,G,\mathcal{L}|_U)$ whose kernel is the subcomplex $C_j(X,G,\mathcal{L})_Y \subset C(X,G,\mathcal{L})$ of cycles supported on $Y$. If $Y$ is of codimension $c$, then the morphism of complexes $C_j(X,G,\mathcal{L})\to C_j(U,G,\mathcal{L}|_U)$ is thus an isomorphism in degrees $\leq c-1$. The result follows.
\end{proof}

\begin{rem}
Note that the restriction map $\widetilde{CH}^i(X) \to \widetilde{CH}^i(U)$ attached to an open immersion $U \to X$ is not surjective in general.
\end{rem}

\section{Euler classes and Thom classes}
\label{s:eulerthom}
In this section, we do several things.  Our first goal is to define a ``universal" Chow-Witt Euler class.  To this end, Subsection \ref{ss:eulerfunctoriality} establishes functoriality of the Euler class for the pullback studied in the previous section.  In Subsection \ref{ss:chowwittstabilization}, we show how to extend the Chow-Witt groups to the infinite Grassmannian $Gr_n$, which is usually presented as a colimit of smooth schemes; this discussion is very similar to that study of Chow groups of classifying spaces \`a la Totaro and Edidin-Graham \cite{EdidinGraham}.  Then, in Subsection \ref{ss:universalchowwitteuler}, we show how to use the functoriality of Euler classes to define a universal Chow-Witt Euler class on the Grassmannian.  In Subsection \ref{ss:obstructionsandactions}, we recall the definition of the obstruction theoretic Euler class and show that it lives in a group that is canonically isomorphic to the group housing the universal Chow-Witt Euler class.  In particular, it makes sense to compare the Chow-Witt and obstruction theoretic Euler classes. Finally in Subsection \ref{ss:beyondsmooth} we extend some ideas of sheaf cohomology to spaces that are not smooth schemes and use this to identify the Euler class in terms of a suitably reinterpreted Thom class.

\subsection{Functorial properties of Euler classes}
\label{ss:eulerfunctoriality}
If $X$ is a smooth $k$-scheme, suppose $\xi: \mathcal{E} \to X$ is a vector bundle of rank $r$, and $s_0: X \to \mathcal{E}$ is the zero section.  The Chow-Witt Euler class is defined in the introduction via the formula $e_{cw}(\xi) := (\xi^*)^{-1}(s_0)_* \langle 1 \rangle \in \widetilde{CH}^r(X,\operatorname{det}(\xi)^{\vee})$. Observe that since the zero section of a rank $0$ bundle over a smooth scheme $X$ is simply the identity map $id_X$, the Euler class of a rank $0$ bundle over a smooth scheme is simply $\langle 1 \rangle \in \widetilde{CH}^0(X)$.

In general, The Euler class is only shown to be functorial for flat morphisms of schemes in \cite[Th{\'e}or{\`e}me 13.3.1]{FaselChowWitt}. We now prove a more general result using the Base change theorem \ref{thm:basechange}.

\begin{prop}
\label{prop:functorialityofeulerclasses}
If $f: X \to Y$ is any morphism of smooth schemes, and $\xi: \mathcal{E} \to Y$ is a vector bundle on $Y$, then $f^*e_{cw}(\xi) = e_{cw}(f^*\xi)$.
\end{prop}

\begin{proof}
As usual, we factor $f$ as
\[
X \stackrel{\Gamma}{\longrightarrow} X \times Y \stackrel{p_2}{\longrightarrow} Y
\]
where $\Gamma$ is given by the graph of $f$ (in particular a closed immersion) and the second morphism is a projection (in particular flat).  Observe that $e_{cw}(p_2^* \xi) = p_2^* e_{cw}(\xi)$ by \cite[Th{\'e}or{\`e}me 13.3.1]{FaselChowWitt}.  Thus, we have reduced the problem to establishing the assertion for $f$ a regular embedding.

Suppose $f: X \to Y$ is a closed immersion of smooth schemes.  In that case there is a Cartesian square of the form
\[
\xymatrix{
f^* \mathcal{E} \ar[r]^{f'}\ar[d]_{\xi'} &  \mathcal{E} \ar[d]^{\xi} \\
X \ar[r]_f& Y.
}
\]
Let $s_0': X \to f^* \mathcal{E}$ and $s_0: Y \to \mathcal{E}$ be the zero sections of the two vector bundles in question.  Now, the Euler class of $f^* \mathcal{E}$ is given by $(\xi'^*)^{-1} (s_0')_* \langle 1 \rangle$ and it suffices to show that $f^* (\xi^*)^{-1} (s_0)_* \langle 1 \rangle = (\xi'^*)^{-1} (s_0')_* \langle 1 \rangle$.

To this end, note that the corresponding diagram
\[
\xymatrix{
f^* \mathcal{E} \ar[r]^{f'} &  \mathcal{E}  \\
X \ar[u]^{s_0'} \ar[r]_f& Y\ar[u]_{s_0}.
}
\]
is also Cartesian.  Now, all the morphisms in this diagram are regular embeddings.  Therefore, as observed in \cite[Example 6.3.2]{Fulton} the two possible orientations of the diagram give rise to the same excess bundle.  Now, the normal bundle to the zero section of a vector bundle is simply the original vector bundle, and therefore we conclude that the excess bundle is a rank $0$ vector bundle.  Therefore, Theorem \ref{thm:basechange} applied to the above diagram yields the formula $(f')^*(s_0)_* = (s_0')_*f^*$.

Since $\xi f' = f \xi'$ we see that $(f')^* \xi^* = (\xi')^* f^*$.  Since $\xi^*$ and ${\xi'}^*$ are both isomorphisms, we conclude that $f^* (\xi^*)^{-1} = ({\xi'}^*)^{-1}(f')^*$.  Therefore,
\[
f^* (\xi^*)^{-1} (s_0)_* = ({\xi'}^*)^{-1}(f')^* (s_0)_* = ({\xi'}^*)^{-1}(s_0')_*f^*.
\]
Now, using \cite[Propositions 6.8 and 7.2]{FaselChowWittRing} we observe that $\langle 1 \rangle$ is a unit for the Chow-Witt ring and $f^*$ is a ring homomorphism, i.e., $f^* \langle 1 \rangle = \langle 1 \rangle$.
\end{proof}

\subsection{Stabilization of (twisted) Chow-Witt groups}
\label{ss:chowwittstabilization}
Suppose $G$ is a linear algebraic group over a perfect field $k$, $(\rho,V)$ is a faithful, finite dimensional $k$-rational representation of $G$.  For an integer $n > 0$, we can consider the finite dimensional $k$-rational representation $V^{\oplus \dim V + n}$ of $G$.  There is a canonical $G$-equivariant isomorphism ${\mathbb A}(V^{\oplus \dim V + n}) \cong {\mathbb A}(V)^{\times \dim V + n}$ where $G$ acts diagonally on the product.  Let $V_n \subset {\mathbb A}(V)^{\times \dim V + n}$ be the (maximal) open subscheme on which $G$ acts (scheme-theoretically) freely.

Fix a smooth $G$-scheme $X$ and assume that the contracted product scheme $X \times^G V_n$ exists as a smooth scheme; this latter assumption holds if, e.g., $X$ is $G$-quasi-projective.  The inclusion ${\mathbb A}(V)^{\times \dim V + n} \to {\mathbb A}(V)^{\times \dim V + n+1}$ as the first $\dim V + n$ factors gives rise to a $G$-equivariant morphism $V_n \to V_{n+1}$ and therefore to bonding maps
\[
b_n: X \times^G V_n \to X \times^G V_{n+1}.
\]
Set $X_G(\rho) := \colim_n b_n$ and write $BG(\rho)$ for $X_G(\rho)$ if $X = \Spec k$.  We will refer to the spaces $X \times^G V_{n}$ as {\em finite-dimensional approximations to $X_G(\rho)$}.

The inclusion $V_n \hookrightarrow V_{n+1}$ fits into a commutative diagram of the form:
\[
\xymatrix{
V_n \ar[d]_{id \times 0}\ar[r] & V_{n+1} \ar[d] \\
V_n \times {\mathbb A}(V) \ar[r]_{id} & V_n \times {\mathbb A}(V);
}
\]
here, $0$ is the canonical map $\Spec k \to {\mathbb A}(V)$ corresponding to the inclusion of $0$, the right vertical morphism is an open immersion whose complement has codimension tending to $\infty$ as $n \to \infty$.

If $X$ is a smooth $G$-scheme as above there is a diagram of the form
\begin{equation}
\label{eqn:fundamentalfactorization}
\xymatrix{
X \times^G V_n \ar[r]^{b_n}\ar[d] & X \times^G V_{n+1} \ar[d]\\
X \times^G (V_n \times {\mathbb A}(V)) \ar[r]_{id} & X \times^G (V_n \times {\mathbb A}(V)).
}
\end{equation}
Using faithfully flat descent, we see that the projection map $V_n \times {\mathbb A}(V) \to V_n$ makes $X \times^G (V_n \times {\mathbb A}(V)) \to X \times^G V_n$ a vector bundle and that the left vertical map is simply the morphism corresponding to the zero section.  Furthermore, the right vertical map is an open immersion whose codimension tends to $\infty$ as $n \to \infty$.

If $U \subset X$ has complement of codimension $\geq d$, then the induced map $CH^i(X) \to CH^i(U)$ is an isomorphism for $i \leq d-1$ via the localization sequence for Chow groups.  As a consequence, $\lim_n CH^i(X \times^G V_n)$ is well-defined and we set $CH^i(X_G(\rho)) := \lim_n CH^i(X \times^G V_n)$.  In particular, taking $i = 1$, we can define $Pic(X_G(\rho))$.  The following result is a special case of \cite[Definition-Proposition 1 p. 599]{EdidinGraham}.

\begin{lem}
\label{lem:stabilizationofpicardgroups}
The group $Pic(X_G(\rho))$ is independent of the choice of $\rho$.
\end{lem}

Now, fix a representation $\rho$.  Since $Pic(X_G(\rho))$ is well-defined, we can fix a line bundle $\L$ on $X_G(\rho)$ representing any element of $Pic(X_G(\rho))$.  More precisely, such a line bundle can be represented by a sequence of line bundles $\L_n$ on $X \times^G V_n$ together with specified isomorphisms $b_n^* \L_{n+1} \isomt \L_n$.  Then, for any integer $i$, the groups $\widetilde{CH}^i(X \times^G V_n,\L_n)$ are defined.

\begin{thm}
\label{thm:stabilization}
With $\rho, V_n$ and $\L_n$ as described in the preceding paragraphs, for any given integer $i$, the groups $\widetilde{CH}^i(X \times^G V_n,\L_n)$ stabilize, i.e., there exists an integer $N$ such that for every integer $r \geq 0$, the pullback map $\widetilde{CH}^i(X \times^G V_N,\L_N) \to \widetilde{CH}^i(X \times^G V_{N+r},\L_{N+r})$ is an isomorphism.
\end{thm}

\begin{proof}
Consider Diagram \ref{eqn:fundamentalfactorization}.  We can, without loss of generality, assume that the line bundle $\L_{n+1}$ extends from $X \times^G V_{n+1}$ to a line bundle $\L_{n+1}'$ on $X \times^G (V_n \times {\mathbb A}(V))$ in a fashion such that the restriction of this extended bundle to $X \times^G V_n$ coincides with $\L_n$.  Indeed, if $U \subset Y$ is an open immersion of smooth schemes whose complement has codimension $\geq 2$, then the restriction map on categories of line bundles is fully-faithful and essentially surjective and therefore an equivalence of categories.

Combining these observations there is a corresponding commutative diagram of Chow-Witt groups of the form:
\[
\xymatrix{
\widetilde{CH}^i(X \times^G (V_n \times {\mathbb A}(V)),\L_{n+1}') \ar[d]\ar[r] & \widetilde{CH}^i(X \times^G V_{n+1},\L_{n+1}) \\
\widetilde{CH}^i(X \times^G V_n,\L_n) \ar[ur]_{b_n^*} &
}
\]
The vertical map is pullback along the zero-section of a vector bundle and hence by homotopy invariance for twisted Chow-Witt groups is an isomorphism.  The horizontal morphism has complement of codimension that tends to $\infty$ as $n \to \infty$.  For fixed $i$ and $n$ sufficiently large, Lemma \ref{lem:stabilization} implies this map is an isomorphism.  Therefore, the pullback map $b_n^*$ is necessarily an isomorphism as well.
\end{proof}

Because of the Theorem \ref{thm:stabilization}, the following notation makes sense.

\begin{notation}
Set $\widetilde{CH}^i(X_G(\rho),\L) := \lim_n \widetilde{CH}^i(X \times^G V_n,\L_n)$.
\end{notation}

\begin{rem}
The (twisted) Chow-Witt groups of $X_G(\rho)$ can be seen to be independent of $\rho$.  One way to do this is as follows: the space $X_G(\rho) := \colim_n b_n$ has $\aone$-homotopy type independent of $\rho$; when $X = \Spec k$ this independence result is established in \cite[\S 4.2, esp. Remark 4.2.7]{MV}.  In general, this independence statement can be established by the ``Bogomolov double fibration trick"; see, e.g, \cite[p. 5]{Totaro} or \cite[Definition Proposition 1]{EdidinGraham}. If $\rho$ and $\rho'$ are two faithful representations on vector spaces $V$ and $V'$, we can consider the representation $\rho \oplus \rho'$ on $V \oplus V'$.  There are induced maps $X_G(\rho \oplus \rho') \to X_G(\rho)$ (and corresponding maps of finite dimensional approximations) that one can check are $\aone$-weak equivalences.  The induced maps on finite-dimensional approximations induce isomorphisms on (twisted) Chow-Witt groups in a range of dimensions that tends to infinity as the parameter $n$ in the definition of $X_G(\rho)$ tends to infinity.  In the sequel, we will take $G = GL_n$ (or $SL_n$) and $\rho$ to be the standard $n$-dimensional representation and the independence of the choice of representation will be irrelevant for our purposes here.
\end{rem}

\subsection{The universal Chow-Witt Euler class}
\label{ss:universalchowwitteuler}
Apply Theorem \ref{thm:stabilization} in the case $G = GL_n$, $X = \Spec k$ and $(V,\rho)$ the standard $n$-dimensional representation of $GL_n$.  In that case, $V_N$ can be identified as the open subscheme of the affine space associated with $n \times n+N$-dimensional matrices whose closed complement is defined by the condition that matrices have rank $\leq n-1$.  The quotients $V_N/GL_n$ can be identified as finite dimensional Grassmannian varieties $Gr_{n,n+N}$.  In this case, we write $Gr_n := \colim_N Gr_{n,n+N}$ for $BGL_n(\rho)$.  By the discussion of the previous section, for any line bundle $\L$ on $Gr_n$, we can now speak of $\widetilde{CH}^i(Gr_n,\L)$.

Each Grassmannian $Gr_{n,n+N}$ carries a universal vector bundle $\gamma_{n,N}$ whose fiber over a point is precisely the hyperplane corresponding to that point.  There is a commutative diagram (more precisely, a Cartesian square)
\[
\xymatrix{\mathcal{V}_{n,N}\ar[r]\ar[d]_-{\gamma_{n,N}} & \mathcal{V}_{n,N+1}\ar[d]^-{\gamma_{n,N+1}} \\
Gr_{n,n+N}\ar[r] & Gr_{n,n+N+1}}
\]
where the horizontal arrows are closed immersions.  Set $\mathcal{V}_{n} := \colim_N \mathcal{V}_{n,N}$ be the colimit of the above maps and let $\gamma_n:\mathcal{V}_n\to Gr_n$ be the induced morphism.  Consider the bonding map $b_n: Gr_{n,n+N} \to Gr_{n,n+N+1}$.  Because $b_n^*(\gamma_{n,N+1}) = \gamma_{n,N}$, Proposition \ref{prop:functorialityofeulerclasses} guarantees that $b_n^* e_{cw}(\gamma_{n,N}) = e_{cw}(b_n^* \gamma_{n,N+1})$.  Thus, the sequence of elements $e_{cw}(\gamma_{n,N})$ yields a well-defined element of $\lim_N \widetilde{CH}^i(X \times^G V_N,\det \gamma_{n,N})$ that we will call $e_{cw}(\gamma_n)$.  Since by Theorem \ref{thm:stabilization} the pullback maps are isomorphisms for $n$ sufficiently large, the Euler class $e_{cw}(\gamma_n)$ is represented by $e_{cw}(\gamma_{n,N})$ for all $n$ sufficiently large.  We summarize this in the following result.

\begin{cor}
\label{cor:universalchowwitteulerclass}
If $X = Gr_n$ and $\gamma_n: \mathcal{V}_n \to Gr_n$ is the universal rank $n$ vector bundle on $Gr_n$, then there is a unique element $e_{cw}(\gamma_n) \in \widetilde{CH}^n(Gr_n,\det(\gamma_n)^{\vee})$ whose restriction to the group $\widetilde{CH}^n(Gr_{n,n+N},\det(\gamma_{n,N})^{\vee})$ is $e_{cw}(\gamma_{n,N})$ for every integer $N$ sufficiently large.
\end{cor}

\begin{rem}
\label{rem:slcase}
Let $BG$ be the standard simplicial classifying space of the (Nisnevich) sheaf of groups $G$, e.g., as discussed in \cite[\S 4]{MV}.  This construction is functorial, and the homomorphism $\det: GL_n \to \gm{}$ induces a canonical morphism $BGL_n \to B\gm{}$.  By \cite[\S 4 Proposition 1.15]{MV}, $BGL_n$ is a classifying space for vector bundles: if $X$ is a smooth $k$-scheme, then there is a bijection between simplicial homotopy classes of maps with source $X$ and target $BGL_n$ and isomorphism classes of rank $n$ vector bundles on $X$ (this description actually works more generally, and we will use this observation momentarily).  If $X \to BGL_n$ is a simplicial homotopy class corresponding to a vector bundle $\xi$, then the composite map $X \to BGL_n \to B\gm{}$ corresponds to the $\gm{}$-torsor $(\det \xi^{\vee})^{\circ}$.

There is a simplicial homotopy class of maps $Gr_n \to BGL_n$ classifying the universal vector bundle $\gamma_n$ over $Gr_n$.  By \cite[\S 4 Proposition 3.7]{MV}, this map is an $\aone$-weak equivalence.  As a consequence, the composite map $Gr_n \to BGL_n \to B\gm{}$ corresponds to a $\gm{}$-torsor over $Gr_n$.  Since $Gr_n$ is the colimit of $Gr_{n,n+N}$, the composite maps $Gr_{n,n+N} \to BGL_n \to B\gm{}$ correspond to considering the $\gm{}$-torsors $(\det \gamma_{n,N}^\vee)^{\circ}$.  Using this, the composite map $Gr_n \to BGL_{n} \to B\gm{}$ is represented by the colimit of the spaces $(\det \gamma_{n,N}^\vee)^0$ to $BGL_n$ to $BGL_n$.  This can all be seen differently as follows.

The total space of this $\gm{}$-torsor is a model for the classifying space $BSL_n$; this corresponds to taking the standard representation of $GL_n$ and viewing it as a representation of $SL_n$ by restriction.  It is straightforward to check that this model of $BSL_n$ corresponds to the complement of the zero section in the total space of the {\em dual} of the determinant of $\gamma_n$ over $Gr_n$ (the dual comes because the total space of a vector bundle is the spectrum of the symmetric algebra of the dual bundle).  It is straightforward to check that this space is precisely $BSL_n(\rho)$ for $\rho$ the standard $n$-dimensional representation of $SL_n$.

The pullback of $\gamma_n$ to this model of $BSL_n$ therefore is a universal vector bundle with a specified trivialization of the determinant, i.e., an oriented vector bundle.  Abusing notation, also denote by $\gamma_n$ this oriented vector bundle.  The discussion above gives a canonical element $e_{cw}(\gamma_n) \in \widetilde{CH}^n(BSL_n)$ (now there is no twist!).
\end{rem}

\subsection{Obstruction groups and actions}
\label{ss:obstructionsandactions}
The obstruction theoretic Euler class is defined using $k$-invariants arising in a Moore-Postnikov factorization, which we describe in a fashion slightly different from the introduction.  Let $BGL_n$ be the usual simplicial classifying space as described, e.g., in \cite[\S 4]{MV}.  In that case, functoriality of the classifying space construction applied to the inclusion map $GL_{n-1} \to GL_{n}$ sending an invertible $(n-1) \times (n-1)$ matrix $M$ to the block-diagonal matrix $diag(M,1)$ yields a morphism $BGL_{n-1} \to BGL_{n}$.

The space $BGL_n$ is $\aone$-weakly equivalent to the space $Gr_n$ mentioned in the introduction by \cite[\S 4 Proposition 3.7]{MV} (see Remark \ref{rem:slcase} for some context), and there is an $\aone$-fiber sequence of the form
\[
{\mathbb A}^{n} \setminus 0 \longrightarrow BGL_{n-1} \longrightarrow BGL_{n}.
\]
The obstruction groups that arise by means of the Moore-Postnikov factorization of this map have coefficients in higher $\aone$-homotopy sheaves of ${\mathbb A}^{n} \setminus 0$ twisted by an action of $\bpi_1^{\aone}(BGL_{n+1})$.  The space $BGL_n$ is not $\aone$-$1$-connected and our goal here is to identify this action and this group; the proofs of these results appear in \cite{AsokFaselA3minus0}, but are scattered throughout the paper, so for the reader's convenience we reintroduce all the necessary terminology here.

The exact sequence $SL_n \to GL_n \to \gm{}$ of Nisnevich sheaves of groups yields a simplicial fiber sequence $BSL_n \to BGL_n \to B\gm{}$ which, since $B\gm{}$ is $\aone$-local, is also an $\aone$-fiber sequence and by shifting, a fiber sequence of the form $\gm{} \to BSL_n \to BGL_n$.  The map $SL_n \to GL_n$ defines a map $ESL_n \to EGL_n$ that is $SL_n$-equivariant and consequently yields a map $BSL_n \to EGL_n/SL_n$; because $ESL_n$ and $EGL_n$ are simplicially contractible, this map is a simplicial weak equivalence.  Thus, up to simplicial weak equivalence we can take
\[
GL_n/SL_n \longrightarrow EGL_n / {SL_n} \longrightarrow EGL_n / {GL_n}.
\]
as a model for this $\aone$-fiber sequence.

The determinant yields an isomorphism $GL_n/SL_n \cong \gm{}$ that makes the above sequence into a $\gm{}$-torsor.  Consider the splitting $\gm{} \to GL_n$ sending $t$ to $diag(t,1,\ldots,1)$.  The conjugation action of $\gm{}$ on $GL_n$ induced by this splitting yields an action of $\gm{}$ on $SL_n$ by restriction and the maps $EGL_n / SL_n \to EGL_n / GL_n$ are then equivariant for this action of $\gm{}$.  In particular, we can identify the standard action of $\gm{}$ on $EGL_n/SL_n$ coming from its identification as a $\gm{}$-torsor as that induced by the conjugation action just specified.

Morel showed that $BSL_n$ is $\aone$-$1$-connected (see, e.g., \cite[Example 2.5]{AsokFaselSpheres}) and therefore $EGL_n/SL_n$ is $\aone$-$1$-connected as well.  This discussion above shows that $\bpi_1^{\aone}(BGL_n)$ is then isomorphic to $\gm{}$.  On the other hand, since $\gm{}$ is strongly $\aone$-invariant, the map $EGL_n/SL_n \to BGL_n$ an $\aone$-covering space by \cite[Lemma 7.5(1)]{MField}, thus by \cite[Theorem 7.8]{MField} it is necessarily a universal $\aone$-covering space.  Thus, we conclude that the action of $\gm{}$ on $EGL_n/SL_n$ by conjugation (described above) is a model for the action of $\gm{} = \bpi_1^{\aone}(BGL_n)$ on the $\aone$-universal cover of $BGL_n$ by ``deck transformations."

Next, we use the above results to obtain a description of the action of $\gm{} = \bpi_1^{\aone}(BGL_n)$ on $\bpi_i^{\aone}({\mathbb A}^n \setminus 0)$ mentioned at the beginning of this section.  To this end, consider the commutative diagram of inclusions
\[
\xymatrix{
SL_{n-1} \ar[r]\ar[d] & SL_n \ar[d] \\
GL_{n-1} \ar[r] & GL_n,
}
\]
where the horizontal maps are those sending a matrix $M$ to the block matrix $diag(M,1)$.  The group $\gm{}$ acts on all of the groups in the diagram by conjugation by $diag(t,1,\ldots,1)$ and with respect to these actions all the morphisms are equivariant.

The induced map of quotients $SL_n/SL_{n-1} \to GL_n/GL_{n-1}$ is an isomorphism that is $\gm{}$-equivariant for the actions of $\gm{}$ induced on the quotients.  Consider the map $GL_n \to {\mathbb A}^n \setminus 0$ given by projection onto the last column.  If we equip ${\mathbb A}^n \setminus 0$ with the action of $\gm{}$ given in coordinates by
\begin{equation}
\label{eqn:gmaction}
t \cdot (x_1,\ldots,x_n) = (tx_1,\ldots,x_n),
\end{equation}
then the projection map factors through an $\aone$-weak equivalence $GL_n/GL_{n-1} \to {\mathbb A}^n \setminus 0$ that is also $\gm{}$-equivariant.

Observe that there is a commutative diagram of fiber sequences
\[
\xymatrix{
{\mathbb A}^n \setminus 0 \ar[r]\ar@{=}[d] & BSL_{n-1} \ar[r]\ar[d] & BSL_{n} \ar[d] \\
{\mathbb A}^n \setminus 0 \ar[r] & BGL_{n-1} \ar[r] & BGL_n
}
\]
The sequence in the top row admits a model via the simplicial fiber sequence
\[
SL_n/SL_{n-1} \longrightarrow BSL_{n-1} \longrightarrow BSL_{n}
\]
and this sequence of maps is $\gm{}$-equivariant with respect to the actions mentioned above.  Using the identifications above, the action of $\bpi_1^{\aone}(BGL_n)$ on the higher $\aone$-homotopy groups of ${\mathbb A}^n \setminus 0$ is induced by the $\gm{}$-action on ${\mathbb A}^n \setminus 0$ described in \ref{eqn:gmaction}.

\begin{prop}
\label{prop:gmaction}
The $\gm{}$-action on $\bpi_{n-1}^{\aone}({\mathbb A}^n \setminus 0) \cong \K^{MW}_n$ coincides with that described in \textup{\S \ref{ss:contractionsandactions}}.
\end{prop}

\begin{proof}
This result is contained in the discussion of \cite[\S 6.2]{AsokFaselA3minus0}.
\end{proof}

As above, write $\gamma_n: \mathscr{V}_n \to Gr_n$ for the universal vector bundle over the Grassmannian.  We can replace $BGL_n$ functorially by a simplicially fibrant model for which we shall write $BGL_n^f$.  In that case, the simplicial homotopy class in $[Gr_n,BGL_n]$ corresponding to $\gamma_n$ is represented by an actual map $\gamma_n: Gr_n \to BGL_n^f$; as noted before, this map is an $\aone$-weak equivalence by \cite[\S 4 Proposition 3.7]{MV}.  We can thus view the obstruction theoretic Euler class as the first non-trivial $k$-invariant attached to the Moore-Postnikov factorization of the map $Gr_{n-1} \to Gr_n$, which fits into the $\aone$-fiber sequence
\[
{\mathbb A}^n \setminus 0 \longrightarrow Gr_{n-1} \longrightarrow Gr_n.
\]

If $X$ is a smooth $k$-scheme, and $\xi: \mathcal{E} \to X$ is a rank $n$ vector bundle on $X$.  There is an induced class $\xi \in [X,BGL_n]_{\aone}$.  The map $BGL_n \to B\gm{}$ defines a class in $[X,B\gm{}]_{\aone}$ that corresponds to the line bundle $\det \xi$.  The group housing the primary obstruction is described in \cite[\S 6.1-6.2]{AsokFaselA3minus0}.  In particular, if $\xi$ has rank $n$, then the primary obstruction defines a class
\[
e_{obs}(\xi) \in H^n(X,\K^{MW}_n(\det \xi)).
\]
We use Theorem \ref{thm:identification1} to identify $\widetilde{CH}^n(X,\det \xi) \cong H^n(X,\K^{MW}_n(\det \xi))$, and under this isomorphism, we can compare the obstruction-theoretic Euler class and the Chow-Witt Euler class.

\begin{ex}
The class $e_{obs}(\xi)$ defined above is functorial with respect to pullbacks.  Recall that the groups $\widetilde{CH}^n(Gr_n,\L)$ are defined with respect to finite dimensional approximations that stabilize \ref{thm:stabilization}.  Pulling back the universal obstruction classes with respect the classifying maps $Gr_{n,n+N} \to BGL_n$, we obtain a sequence of classes $e_{obs}(\gamma_{n,N}) \in \widetilde{CH}^n(Gr_{n,n+N},\det \gamma_{n,N}^{\vee})$ such that $b_N^* e_{obs}(\gamma_{n,N+1}) = e_{obs}(\gamma_{n,N})$.  As a consequence, there is a unique class in $e_{obs}(\gamma_n) \in \widetilde{CH}^n(Gr_n,\det \gamma_n^{\vee})$ that is given by $e_{obs}(\gamma_{n,N})$ in any ``sufficiently large finite dimensional approximation."
\end{ex}

\subsection{Cohomology of Thom spaces and Thom classes}
\label{ss:beyondsmooth}
The goal of this section is to identify the Euler class in terms of Thom isomorphisms.  To this end, and to facilitate geometric arguments, we use some extensions of sheaf cohomology to ``spaces" in the sense of Morel-Voevodsky (i.e., simplicial Nisnevich sheaves).

\subsubsection*{D\'evissage and Thom classes}
Now, recall that the map $(s_0)_*$ is actually the composite of two maps.  Indeed, $\widetilde{CH}^i(X,\mathcal{L})$ is, as mentioned above, defined as the cohomology of a Gersten-type complex $C(X,G,\mathcal{L})$.  If $\mathcal{E}^{\circ}$ is the complement of the zero section of our vector bundle, then there is an exact sequence of complexes of the form
\begin{equation}
\label{eqn:localization}
0 \longrightarrow C(\mathcal{E},G,\mathcal{L})_X \longrightarrow C(\mathcal{E},G,\mathcal{L}) \longrightarrow C(\mathcal{E}^{\circ},G,\mathcal{L}|_{\mathcal{E}}) \longrightarrow 0.
\end{equation}
As mentioned in \cite[Remarque 10.4.8]{FaselChowWitt}, there is a d\'evissage isomorphism of relative degree $r=\mathrm{rank}(\mathcal E)$
\[
C(X,G,(s_0)^*\mathcal{L} \tensor \det \xi) \isomto C(\mathcal{E},G,\mathcal{L})_X,
\]
and the map $(s_0)_*$ on Chow-Witt groups is induced by the composite of the d\'evissage isomorphism and the extension of support homomorphism. Taking $\mathcal L=\xi^*\det \xi^\vee$, the d\'evissage homomorphism yields an isomorphism
\[
C_0(X,G) \isomto C_r(\mathcal{E},G,\xi^*\det \xi^\vee)_X,
\]

\begin{defn}[Thom class]
\label{defn:thomclass}
Given a smooth scheme $X$ and a vector bundle $\xi: \mathcal{E} \to X$, there is a unique class $t_{\xi} \in H^r_X(\mathcal E,G,\xi^* \det \xi^{\vee})$ that corresponds under the d\'evissage isomorphism to the class $\langle 1 \rangle \in \widetilde{CH}^0(X)$; we call this class the {\em Thom class} of $\xi$.
\end{defn}

\begin{rem}
It follows immediately from the definition that we can write $e_{cw}(\xi) = (\xi^*)^{-1}\mathrm{ext}(t_{\xi})$, where
\[
\mathrm{ext}: H^r_X(\mathcal E,G,\xi^* \det \xi^{\vee}) \longrightarrow \widetilde{CH}^r(\mathcal{E},\xi^* \det \xi^{\vee})
\]
is the extension of support homomorphism.
\end{rem}

\begin{ex}
Suppose $k$ is our fixed base field.  In the special case where $X = \Spec k$ we consider a geometric vector bundle $\xi: {\mathbb A}(V) \to \Spec k$, where $V$ is an $r$-dimensional $k$-vector space.  Set $\omega = \Lambda^r V$, which is a $1$-dimensional $k$-vector space.  In that case, we obtain a Thom class $t_{\xi} \in H^r_{\{0\}}({\mathbb A}(V),G,\xi^* \omega)$ as the image of $\langle 1 \rangle \in GW(k)$ under the d\'evissage isomorphism.  The ring structure on Chow-Witt groups gives $H^r_{\{0\}}({\mathbb A}(V),G,\xi^* \omega)$ the structure of a $GW(k)$-module of rank $1$.  We already know that this group is, by the d\'evissage isomorphism, identified with $GW(k)$ and it is actually free of rank $1$.   The Thom class gives a prescribed identification $H^r_{\{0\}}({\mathbb A}(V),G,\xi^* \omega) \cong GW(k)$, i.e., it gives a generator for this free rank $1$ module.
\end{ex}

\begin{prop}[Functoriality of Thom classes]
\label{prop:functorialityofthomclasses}
Fix a field $k$.  Suppose $X$ is a smooth $k$-scheme, $\xi: \mathcal{E} \to X$ is a rank $r$ vector bundle.  If $x: \Spec k \to X$ a $k$-rational point, and $\xi': \mathcal{E}|_x \to \Spec k$ is the induced vector bundle, then $x^* t_{\xi} = t_{\xi'}$.
\end{prop}

\begin{proof}
Let $x'$ be the inclusion of the fiber map $\mathcal{E}|_x \to \mathcal{E}$.  Consider the diagram
\[
\xymatrix{
\widetilde{CH}^0(X) \ar[r]\ar[d]^{x^*} & \widetilde{CH}(\mathcal{E},\xi^* \det \mathcal{E}) \ar[d]^{}\\
\widetilde{CH}^0(\Spec k) \ar[r] & \widetilde{CH}(\mathcal{E}|_x,\xi'^*\det \mathcal{E}|_x).
}
\]
This diagram commutes by functoriality of the d\'evissage morphism in transversal squares.  Moreover, we have $x^*(\langle 1\rangle)=\langle 1\rangle$ since $x^*$ is a ring homomorphism.
\end{proof}

\subsubsection*{Sheaf cohomology of spaces}
If $\mathscr{X}$ is a motivic space, and $\mathbf{A}$ is a strictly $\aone$-invariant sheaf of abelian groups, we can define $H^n(\mathscr{X},\mathbf{A}) = [\mathscr{X},K(\mathbf{A},n)]_{\aone}$.  Alternatively, since $K(\mathbf{A},n)$ is $\aone$-local, this construction can be made in the Nisnevich simplicial homotopy category.

The (pointed) Nisnevich simplicial homotopy category is the left Bousfield localization of the category of (pointed) simplicial Nisnevich sheaves with respect to Nisnevich local weak equivalences, and we will view this category as a model category with respect to the injective local model structure.  Write $\hsnis$ ($\hspnis$) for the associated homotopy category.

In the same vein, consider the category $\mathbf{Ch}^{-}(\Ab_{\Nis})$ of (bounded below) chain complexes of Nisnevich sheaves of abelian groups. We can equip this category with the injective local model structure as well and write $\mathbf{D}^-_{\Nis}$ for the associated homotopy category.

Take a pointed simplicial Nisnevich sheaf $\mathscr{X}$.  Write $\Z(X)$ for the Nisnevich sheaf associated with the presheaf $U \mapsto \Z(X(U))$ (i.e., the free abelian group on the simplicial set $X(U)$).  The base-point determines a morphism $\Z = \Z(\Spec k) \to \Z(\mathscr{X})$ splitting the projection $\Z(\mathscr{X}) \to \Z$.  Write $\tilde{\Z}(X)$ for $\Z(X)/\Z(\ast)$.  We write $C_* \tilde{\Z}(X)$ for the normalized chain complex of the simplicial abelian group $\tilde{\Z}(X)$.  In the opposite direction, we can consider the Eilenberg-MacLane spaces $K(C_*,n) := K(C_*[n])$ associated with a complex $C_*$ of Nisnevich sheaves of abelian groups (truncating the complex to lie in degrees $\geq 0$ if necessary).

The normalized chain complex and Eilenberg-Mac Lane functors pass to a left Quillen adjunction between the associated model categories.  In particular, the adjunction contains functorial bijections for any abelian sheaf $\mathbf{A}$
\[
[\mathscr{X}_+,K(\mathbf{A})]_s \cong Hom_{\mathbf{D}^-_{\Nis}}(C_*\tilde{\Z}(\mathscr{X}),\mathbf{A}),
\]
where we use the subscript $s$ to designate morphisms in the (pointed) Nisnevich simplicial homotopy category.  For this statement, we refer the reader to \cite[\S 2.3 Proposition 3]{DeligneVoevodsky}. Note that these identifications are compatible with the suspension isomorphism, again as discussed in \cite[\S 2.3 p. 364]{DeligneVoevodsky}.

Both the category $\hspnis$ and $\mathbf{Ch}^{-}(\Ab_{\Nis})$ can be $\aone$-localized; the latter is studied in \cite[\S 6.2]{MField}.  In particular, recall that by \cite[Proposition 6.25]{MField} or \cite[\S 2.3 Proposition 4]{DeligneVoevodsky}, the space $K(C_*)$ is $\aone$-local if and only if the complex $C_*$ is $\aone$-local.

Now, if $j: U \to X$ is an open immersion of smooth schemes, then the induced map $C_*\Z(U) \to C_*\Z(X)$ is a cofibration.  It follows from the fact that $C_* \Z (\cdot)$ is a left Quillen functor that $C_*\Z(X/U) = C_*\Z(X)/C_*\Z(U)$ \cite[Proposition 6.4.1]{Hovey}.  In particular, we obtain a distinguished triangle in $\mathbf{D}^-_{\Nis}$.  It follows immediately from this and the adjunction that there is an induced isomorphism
\[
[X/U,K(\mathbf{A},n)]_s \cong \hom_{\mathbf{D}^-_{\Nis}}(C_*\Z(X)/C_*\Z(U),\mathbf{A}[n]).
\]

Now, it follows from the discussion of \cite[\S 2.3 p. 363]{DeligneVoevodsky} that if $X$ is a smooth scheme, then $H^i_{\Nis}(X,\mathbf{A})$ can be computed on the small site of $X$ as well. More precisely, restriction to the small site is the left adjoint of a Quillen adjunction between the ``big" derived category constructed above and the ``small" derived category of Nisnevich sheaves of abelian groups over $X$.  Via this identification, the cohomology on $U$ can also be computed in the small derived category of $X$.  Since the cohomology of an abelian sheaf on $X$ with supports in $Z := X \setminus U$ can be computed in terms of the cone of the restriction map of complexes computing cohomology on $X$ and cohomology on $U$, the next result is an exercise in unwinding definitions.

\begin{prop}
\label{prop:xmodu}
If $\mathbf{A}$ is a strictly $\aone$-invariant sheaf of abelian groups, and $j: U \to X$ is an open immersion of smooth schemes with closed complement $Z$, then there are isomorphisms
\[
[X/U,K(\mathbf{A},n)]_{\aone} \cong H^n_Z(X,\mathbf{A})
\]
functorial in the pair $(Z,X)$.  Moreover, the cofiber sequence $U \to X \to X/U \to \cdots$ yields a long exact sequence in cohomology that, under these identifications, corresponds to the long exact sequence in cohomology with supports.
\end{prop}

\begin{cor}
\label{cor:thomspacesupports}
If $Z \hookrightarrow X$ is a closed immersion of smooth schemes with normal bundle $\nu_{Z/X}$, and $\mathbf{A}$ is a strictly $\aone$-invariant sheaf, then there is a canonical isomorphism
\[
H^n_Z(X,\mathbf{A}) \cong [Th(\nu_{Z/X}),K(\mathbf{A}),n]_{\aone}.
\]
\end{cor}

\begin{proof}
Suppose $\xi: \mathcal{E} \to Y$ is a vector bundle and $\mathcal{E}^{\circ}$ is the complement of the zero section.  Then $Th(\xi)$ is the cone $\mathcal{E}/\mathcal{E}^{\circ}$.  Now, the homotopy purity theorem of \cite[\S 3 Theorem 2.23]{MV} yields a canonical isomorphism in $\hop{k}$ of the form $X/(X-Z) \cong Th(\nu_{Z/X})$.  Then, there are a sequence of isomorphisms:
\[
H^n_Z(X,\mathbf{A}) \cong [X/(X-Z),K(\mathbf{A},n)]_{\aone} \cong [Th(\nu_{Z/X}),K(\mathbf{A},n)]_{\aone} =: H^n(Th(\nu_{Z/X}),\mathbf{A});
\]
The first isomorphism is simply Proposition \ref{prop:xmodu} and the second isomorphism follows from the homotopy purity theorem since $K(\mathbf{A},n)$ is $\aone$-local for any $n$ (which is equivalent to $\mathbf{A}$ being strictly $\aone$-invariant).
\end{proof}

\begin{notation}
\label{notation:thomspaces}
If $\xi: \mathcal{E} \to Y$ is a vector bundle, and $\mathbf{A}$ is a strictly $\aone$-invariant sheaf, we set $H^n(Th(\xi),\mathbf{A}) := [Th(\xi),K(\mathbf{A},n)]_{\aone}$.
\end{notation}

\subsubsection*{Thom classes revisited}
If $\xi: V \to \Spec k$ is a trivial vector bundle of rank $n$, then by \cite[Proposition 2.17 and Corollary 2.18]{MV}, there is a canonical isomorphism ${{\mathbb P}^1}^{\sma n} \cong Th(\xi)$.  Using this identification, together with Notation \ref{notation:thomspaces}, the following result holds.

\begin{lem}
There is a canonical isomorphism $H^n({{\mathbb P}^1}^{\sma n},\K^{MW}_n) \cong \K^{MW}_0(k)$.
\end{lem}

\begin{proof}
Using \cite[Lemma 2.15 and Example 2.20]{MV}, there is an isomorphism ${\pone}^{\sma n} \cong \Sigma^1_s ({\mathbb A}^n \setminus 0)$.  By the suspension isomorphism for homology $H^n({{\mathbb P}^1}^{\sma n},\K^{MW}_n) \cong H^{n-1}({\mathbb A}^n \setminus 0,\K^{MW}_n)$.  The result now follows from \cite[Lemma 4.5]{AsokFaselSpheres} (note that the proof given there replaces $Q_{2n-1}$ by ${\mathbb A}^n \setminus 0$) and \cite[Proposition 2.9]{AsokFaselSpheres}.
\end{proof}

If $X$ is a smooth scheme and $\xi: \mathcal{E} \to X$ is a rank $n$ vector bundle, then given a $k$-rational point $x \in X$, we can restrict $\xi$ to $x$ to obtain a trivial rank $n$ vector bundle over $x$.  Functoriality of the Thom space construction \cite[Lemma 2.1]{VMod2} then defines a map
\[
Th(\xi|_x) \to Th(\xi)
\]
(that is compatible with the purity isomorphism).  A choice of trivialization of $\xi|_x$, i.e., a basis of the $\kappa(x)$-vector space $\xi|_x$ determines an isomorphism $Th(\xi|_x) \cong {{\mathbb P}^1}^{\sma n}$.  In particular, given such a trivialization, there is an induced map
\[
H^n(Th(\xi),\K^{MW}_n) \longrightarrow H^n(Th(\xi|_x),\K^{MW}_n) \isomto H^n({{\mathbb P}^1}^{\sma n}) \isomto \K^{MW}_0(k)
\]
where the middle isomorphism is induced by the choice of trivialization.  For the next statement, recall the definition of the Thom class from Definition \ref{defn:thomclass}.

\begin{lem}
\label{lem:sheafythomclasses}
With respect to the identifications above, the Thom class $t_{\xi}$ of an oriented vector bundle $\xi: \mathcal{E} \to X$ is represented by a unique class in $H^n(Th(\xi),\K^{MW}_n)$ such that for any point $x: \Spec L \to X$, the restriction $x^*t_{\xi}$ is sent to $\langle 1 \rangle \in \K^{MW}_0(L)$ under the isomorphism discussed above.
\end{lem}

\section{A Blakers-Massey theorem in $\aone$-homotopy theory}
\label{s:relativehurewicz}
Fix a field $k$.  Suppose $(\mathscr{X},x)$ and $({\mathscr Y},y)$ are pointed spaces and $f: {\mathscr X} \longrightarrow {\mathscr Y}$ is a morphism of pointed spaces.  The goal of this theorem is to establish a relative Hurewicz theorem comparing the connectivity of the $\aone$-homotopy fiber and $\aone$-homotopy cofiber of the map $f$ under suitable analogs of the classical hypotheses.  This result, which appears as Theorem \ref{thm:relativehurewicz} below, is a slight refinement of F. Morel's relative $\aone$-Hurewicz theorem \cite[Theorem 6.56]{MField} (we have since learned the related results were established by F. Strunk in his thesis \cite[Theorem 2.3.8]{Strunk}).  Subsection \ref{ss:comparison} is devoted to constructing the comparison maps alluded to above and Subsection \ref{ss:connectivity} contains the proof of the main result.

\subsection{Comparison maps}
\label{ss:comparison}
Recall that there exists an endo-functor $Ex_{\aone}$ of $\Spc_k$ and a natural transformation $\theta: Id \to Ex_{\aone}$ such that the induced maps $\mathscr{X} \to Ex_{\aone}(\mathscr{X})$ is an $\aone$-acyclic cofibration and $Ex_{\aone}(\mathscr{X})$ is $\aone$-fibrant (see \cite[\S 2 Definition 3.18, Lemmas 3.20-21]{MV}.  Applying this functor to $f$, we obtain a diagram of the form
\[
\xymatrix{
\mathscr{X} \ar[r]\ar[d]_f & Ex_{\aone}(\mathscr{X}) \ar[d]^{Ex_{\aone}(f)} \\
\mathscr{Y} \ar[r] & Ex_{\aone}(\mathscr{Y}).
}
\]
Now, the morphism $Ex_{\aone}(f)$ is not necessarily an $\aone$-fibration, but by the model category axioms we can functorially factor this morphism as the composite of an $\aone$-acyclic cofibration and an $\aone$-fibration, i.e., there exists a space $\mathscr{Y}'$ and a diagram of the form
\[
\xymatrix{
Ex_{\aone}(\mathscr{X}) \ar[r] \ar[dr]_{Ex_{\aone}(f)} &  \ar[d] \mathscr{Y}' \\
& Ex_{\aone}(\mathscr{Y})
}
\]
where the vertical map is an $\aone$-fibration and the horizontal map is a monomorphism and an $\aone$-weak equivalence.  Note, in particular, that $\mathscr{Y}'$ is also $\aone$-fibrant.  Thus, we can functorially replace $f: \mathscr{X} \to \mathscr{Y}$ by an $\aone$-fibration of $\aone$-fibrant objects without changing the $\aone$-homotopy class of $f$.

Assuming $f$ is an $\aone$-fibration of $\aone$-fibrant objects, we define a space $\mathscr{F}$ as the ordinary fiber of the morphism $f$ over the base-point, i.e., there is a pullback square of the form
\[
\xymatrix{
\mathscr{F} \ar[r] \ar[d] & \ast \ar[d]\\
\mathscr{X} \ar[r]^f & \mathscr{Y}.
}
\]
Because $f$ is an $\aone$-fibration of $\aone$-fibrant objects, it follows that $\mathscr{F}$ is a model for the $\aone$-homotopy fiber of $f$.

Now, if $(\mathscr{Z},z)$ is a pointed space, and $\Delta^1_s$ is the simplicial interval (pointed by $1$), let $C(\mathscr{Z}) := \Delta^1_s \wedge \mathscr{Z}$.  The inclusion $\mathscr{Z} \to \mathscr{Z} \times \Delta^1$ sending $z$ to $(z,0)$ induces a cofibration $\mathscr{Z} \to C(\mathscr{Z})$.  Observe that this construction is functorial in $\mathscr{Z}$ and $C(\mathscr{Z})$ is simplicially weakly equivalent to a point.

Now, consider the map of diagrams
\[
\xymatrix{
C(\mathscr{F}) \ar[d] & \mathscr{F} \ar[l]\ar[r]\ar[d] & \ast \ar[d]\\
C(\mathscr{X})        & \mathscr{X} \ar[l]\ar[r]^f & \mathscr{Y}
}
\]
The pushout of the first row is by definition $\Sigma^1_s \mathscr{F}$, while the pushout of the second row is the model of the (simplicial) homotopy cofiber of the map $f$.  Thus, by functoriality of pushouts (again, as spaces) there is an induced map
\begin{equation}
\label{eqn:comparisonsuspensionA}
\Sigma^1_s \mathscr{F} \longrightarrow C(\mathscr{X}) \cup_{\mathscr{X}} \mathscr{Y}.
\end{equation}
This map corresponds to a map from the simplicial suspension of the $\aone$-homotopy fiber of $f$ to the simplicial homotopy cofiber of $f$.  Therefore, applying the functor $Ex_{\aone}$ yet again, we obtain a map
\begin{equation}
\label{eqn:comparisonsuspensionB}
Ex_{\aone}(\Sigma^1_s \mathscr{F}) \longrightarrow Ex_{\aone}(C(\mathscr{X}) \cup_{\mathscr{X}} \mathscr{Y}).
\end{equation}
Observe that $Ex_{\aone}(C(\mathscr{X}) \cup_{\mathscr{X}} \mathscr{Y})$ is a model for the homotopy cofiber of $f$ computed in the $\aone$-local model structure.

The composite of the canonical map $\Sigma^1_s \mathscr{F} \to Ex_{\aone}(\Sigma^1_s \mathscr{F})$ and the map in the previous paragraph also induces a map
\begin{equation}
\label{eqn:comparisonsuspensionC}
\Sigma^1_s \mathscr{F} \longrightarrow Ex_{\aone}(C(\mathscr{X}) \cup_{\mathscr{X}} \mathscr{Y}).
\end{equation}
We now analyze various adjoints of these maps.

Write ${\mathbf R}\Omega^1_s$ for the (derived) simplicial loops functor.  Under the adjunction of loops and suspension, the map in \ref{eqn:comparisonsuspensionA} corresponds to a map
\begin{equation}
\label{eqn:comparisonloopsA}
\mathscr{F} \longrightarrow {\mathbf R}\Omega^1_s (C(\mathscr{X}) \cup_{\mathscr{X}} \mathscr{Y}).
\end{equation}
Again, applying $Ex_{\aone}$ and observing that $\mathscr{F}$ is already $\aone$-fibrant, we then obtain a map
\begin{equation}
\label{eqn:comparisonloopsB}
\mathscr{F} \longrightarrow Ex_{\aone}{\mathbf R}\Omega^1_s ( C(\mathscr{X}) \cup_{\mathscr{X}} \mathscr{Y}),
\end{equation}
On the other hand, adjunction applied to \ref{eqn:comparisonsuspensionC} yields a map
\begin{equation}
\label{eqn:comparisonloopsC}
\mathscr{F} \longrightarrow {\mathbf R}\Omega^1_s Ex_{\aone}(C(\mathscr{X}) \cup_{\mathscr{X}} \mathscr{Y})
\end{equation}
that compares the $\aone$-homotopy fiber and $\aone$-homotopy cofiber of $f$.

Now, notice that for any simplicially fibrant space $\mathscr{Z}$ the canonical map $\mathscr{Z} \to Ex_{\aone}\mathscr{Z}$ induces a morphism $\Omega^1_s \mathscr{Z} \longrightarrow \Omega^1_s Ex_{\aone} \mathscr{Z}$.  There is then an induced morphism
\[
Ex_{\aone} \Omega^1_s \mathscr{Z} \to Ex_{\aone} \Omega^1_s Ex_{\aone} \mathscr{Z}.
\]
Since $\Omega^1_s Ex_{\aone} \mathscr{Z}$ is already $\aone$-fibrant the map $\Omega^1_s Ex_{\aone} \mathscr{Z} \to Ex_{\aone} \Omega^1_s Ex_{\aone} \mathscr{Z}$ is a simplicial weak equivalence.  It follows from the universal property of $\aone$-localization that there is a simplicial homotopy commutative diagram of comparison maps of the form
\[
\xymatrix{
\mathscr{F} \ar[r] \ar[dr] & Ex_{\aone}{\mathbf R}\Omega^1_s ( C(\mathscr{X}) \cup_{\mathscr{X}} \mathscr{Y}) \ar[d] \\
& {\mathbf R}\Omega^1_s Ex_{\aone}(C(\mathscr{X}) \cup_{\mathscr{X}} \mathscr{Y}).
}
\]

\subsection{Connectivity of fibers and cofibers}
\label{ss:connectivity}
If $(\mathscr{X},x)$ is a pointed space, then by the $i$-th simplicial homotopy sheaf of $\mathscr{X}$ we will mean the Nisnevich sheaf associated with the presheaf on $\Sm_k$ defined by
\[
U \mapsto [S^i_s \wedge U_+, (\mathscr{X},x)]_{\hspnis};
\]
we use the notation $\bpi_i^s(\mathscr{X},x)$ for this sheaf, and for notational convenience we will frequently suppress the base-point from notation.  Similarly, we define $\bpi_i^{\aone}(\mathscr{X},x)$ to be the Nisnevich sheaf associated with the presheaf on $\Sm_k$ defined by
\[
U \mapsto [S^i_s \wedge U_+, {\mathscr{X},x}]_{\hop{k}}.
\]
A space $(\mathscr{X},x)$ is {\em simplicially $n$-connected} (resp. $\aone$-$n$-connected) if $\bpi_i^{s}(\mathscr{X},x)$ (resp. $\bpi_i^{\aone}(\mathscr{X},x)$) is the trivial sheaf for $i \leq n$.  More generally, a morphism $f: \mathscr{X} \to \mathscr{Y}$ is said to be {\em simplicially $n$-connected} if the induced map on simplicial homotopy sheaves is an isomorphism for $i \leq n$ and {\em $\aone$-$n$-connected} if the induced map on $\aone$-homotopy sheaves is an isomorphism for $i \leq n$.  Equivalently, by the long exact sequence in homotopy sheaves associated with a fibration in the corresponding model structure, each of these definitions can be phrased in terms of connectivity of a suitable homotopy fiber.

Suppose now $f: \mathscr{X} \to \mathscr{Y}$ is a pointed map of spaces.  As described in the previous section, we can replace $f$ by an $\aone$-weakly equivalent map that has the property that it is an $\aone$-fibration of $\aone$-fibrant spaces and we can consider the associated comparison map in \ref{eqn:comparisonloopsA}.

\begin{thm}
\label{thm:relativehurewicz}
Suppose $k$ is a perfect field.  Assume $f: \mathscr{X} \to \mathscr{Y}$ is a pointed $\aone$-fibration of $\aone$-fibrant spaces.  Consider the comparison map
\[
\mathscr{F} \longrightarrow {\mathbf R}\Omega^1_s Ex_{\aone} (C(\mathscr{X}) \cup_{\mathscr{X}} \mathscr{Y}).
\]
\begin{itemize}
\item[i)] If $\mathscr{F}$ is $\aone$-$m$-connected ($m \geq 0$) and $\mathscr{X}$ is $\aone$-$1$-connected, then the comparison map is $\aone$-$(m+2)$-connected.
\item[ii)] If $\mathscr{Y}$ is $\aone$-$n$-connected for some $n\geq 2$ and $f$ is $\aone$-$m$-connected for some $m \geq 1$, then the comparison map is $\aone$-$(m+n+1)$-connected.
\end{itemize}
\end{thm}

\begin{proof}
We establish the result under the first set of hypotheses.  First, consider the comparison map
\[
\mathscr{F} \longrightarrow {\mathbf R}\Omega^1_s(C(\mathscr{X}) \cup_{\mathscr{X}} \mathscr{Y})
\]
from \ref{eqn:comparisonloopsA}.  Since $\mathscr{F}$ is $\aone$-fibrant, the hypothesis that $\mathscr{F}$ is $\aone$-$m$-connected is equivalent to the assumption that $\mathscr{F}$ is simplicially $m$-connected, which is also equivalent to the condition that the stalks of $\mathscr{F}$ are $m$-connected simplicial sets.  Likewise, the assumption that $\mathscr{X}$ or $\mathscr{Y}$ has a particular $\aone$-connectivity is equivalent to the assumption that the stalks have the same connectivity.

The functor ${\mathbf R}\Omega^1_s$ is the composite of $\Omega^1_s(\cdot)$ and {\em simplicial} fibrant replacement.  Since taking stalks commutes with formation of colimits, it follows that a stalk of $C(\mathscr{X}) \cup_{\mathscr{X}} \mathscr{Y}$ is canonically isomorphic to the pushout of the stalks of the constituent spaces.  Likewise, it follows from the construction of the simplicial fibrant replacement functor (see, e.g., \cite[\S 2 Theorem 1.66]{MV} and the preceding discussion) that the stalk of the simplicial fibrant replacement of any space is a fibrant replacement (in the model category of simplicial sets) of the stalk of that space.  As a consequence, we conclude that the stalk of ${\mathbf R}\Omega^1_s (C(\mathscr{X}) \cup_{\mathscr{X}} \mathscr{Y})$ is weakly equivalent as a simplicial set to ${\mathbf R}\Omega^1_s$ applied to the stalks $C(\mathscr{X}) \cup_{\mathscr{X}} \mathscr{Y}$.

Now, assume $\mathscr{F}$ is $\aone$-$m$-connected, and $\mathscr{X}$ is $\aone$-$1$-connected.  In that case, the stalks of $f$ satisfy the hypotheses of \cite[Theorem 3.11]{GoerssJardine} and combining the discussion of the previous two paragraphs, we conclude that the comparison map
\[
\mathscr{F} \longrightarrow {\mathbf R}\Omega^1_s (C(\mathscr{X}) \cup_{\mathscr{X}} \mathscr{Y})
\]
is stalkwise $(m+2)$-connected and therefore simplicially $(m+2)$-connected.

Since $\mathscr{F}$ is already fibrant and $\aone$-local, by definition $\bpi_i^s(\mathscr{F}) \cong \bpi_i^{\aone}(\mathscr{F})$.  In particular, the sheaves $\bpi_i^{s}(\mathscr{F})$ are strongly $\aone$-invariant by \cite[Corollary 6.2]{MField}.  Since $m \geq 0$, it follows that $(m+2) \geq 2$ and the discussion of the previous paragraph guarantees that the map
\[
\bpi_i^s(\mathscr{F}) \longrightarrow \bpi_i^s({\mathbf R}\Omega^1_s C(\mathscr{X}) \cup_{\mathscr{X}} \mathscr{Y})
\]
is an isomorphism for $i \leq 2$.  In particular, $\bpi_1^s({\mathbf R}\Omega^1_s C(\mathscr{X}) \cup_{\mathscr{X}} \mathscr{Y})$ is strongly $\aone$-invariant.  In that situation we can apply \cite[Theorem 6.57]{MField} to conclude that the induced map
\[
Ex_{\aone}(\mathscr{F}) \longrightarrow Ex_{\aone}{\mathbf R}\Omega^1_s (C(\mathscr{X}) \cup_{\mathscr{X}} \mathscr{Y})
\]
is simplicially $(m+2)$-connected.

Again using the fact that $\bpi_1^s({\mathbf R}\Omega^1_s C(\mathscr{X}) \cup_{\mathscr{X}} \mathscr{Y})$ is strongly $\aone$-invariant, we conclude that the map
\[
Ex_{\aone}{\mathbf R}\Omega^1_s ( C(\mathscr{X}) \cup_{\mathscr{X}} \mathscr{Y}) \longrightarrow {\mathbf R}\Omega^1_s Ex_{\aone}(C(\mathscr{X}) \cup_{\mathscr{X}} \mathscr{Y})
\]
is a simplicial weak equivalence by \cite[Theorem 6.46]{MField}.  The result then follows by the homotopy commutative diagram of comparison maps at the end of the previous section.

For (ii), we proceed in a completely analogous fashion, except we appeal to the fact that the stalks are simplicially $m+n+1$-connected under these hypotheses, which is a consequence of \cite[Theorem 50]{Mather}.
\end{proof}

\section{Transgression, $k$-invariants and the comparison}
\label{s:transgression}
The goal of this section is to give a nice representative of the $k$-invariant that defines the obstruction theoretic Euler class; we show that the obstruction theoretic Euler class can be described in terms of a ``fundamental class" under ``transgression" (see Lemma \ref{lem:describingkinvariants} and Example \ref{ex:kinvariant}).  This result is the analog of a classical fact relating Moore-Postnikov $k$-invariants and transgressions of cohomology of the fiber (see \cite[Chapter III p. 12]{Thomas} for a general statement about $k$-invariants or \cite[p. 237]{Harper} for the corresponding statement regarding the classical Euler class of an oriented real vector bundle).  In the setting in which we work (simplicial sheaves), it is more or less a question of unwinding definitions; we build on the theory of \cite[VI.5]{GoerssJardine} in the setting of simplicial sets.

Subsection \ref{ss:kinvariants} is devoted to recalling some definitions regarding $k$-invariants in Moore-Postnikov towers in the setting of $\aone$-homotopy theory; the main result is contained in Example \ref{ex:kinvariant} and uses the relative Hurewicz theorem discussed in Section \ref{s:relativehurewicz}.  Subsection \ref{ss:obstructiontheoreticeulerclass} then specializes these results to the case of interest.  Subsection \ref{ss:proofofmaintheorem} then contains the proof of the main result stated in the introduction.  Finally, Subsection \ref{ss:eulercherncompare} contains a refinement of \cite[Proposition 6.3.1]{AsokFaselA3minus0}.

\subsection{On $k$-invariants in Moore-Postnikov towers}
\label{ss:kinvariants}
The Moore-Postnikov tower of a morphism of spaces is constructed in the simplicial homotopy category by sheafifying the classical construction in simplicial homotopy theory \cite[VI.2]{GoerssJardine}.  To perform the same construction in $\aone$-homotopy theory, one applies the construction in the simplicial homotopy category to a fibration of fibrant and $\aone$-local spaces \cite[Appendix B]{MField}.  The description of the $k$-invariants in the Moore-Postnikov factorization is then an appropriately sheafified version of the classical construction.  The first result is an analog of \cite[Lemma VI.5.4]{GoerssJardine} in the context of $\aone$-homotopy theory.

\begin{lem}
\label{lem:describingkinvariants}
Suppose $f: \mathscr{X} \to \mathscr{Y}$ is a morphism of pointed $\aone$-$1$-connected spaces and write $\mathscr{F}$ for the $\aone$-homotopy fiber of $f$.  If $f$ is an $\aone$-$(n-1)$-equivalence for some $n \geq 2$, then for any strictly $\aone$-invariant sheaf $\mathbf{A}$ there are isomorphisms
\[
f^*: H^i(\mathscr{Y},\mathbf{A}) \isomto H^i(\mathscr{X},\mathbf{A}) \text{ if } i < n,
\]
and an exact sequence of the form
\[
0 \longrightarrow H^n(\mathscr{Y},\mathbf{A}) \stackrel{f^*}{\longrightarrow} H^n(\mathscr{X},\mathbf{A}) {\longrightarrow} Hom(\bpi_{n}^{\aone}(\mathscr{F}),\mathbf{A}) \stackrel{\partial}{\longrightarrow} H^{n+1}(\mathscr{Y},\mathbf{A}) \longrightarrow H^{n+1}(\mathscr{X},\mathbf{A}).
\]
Moreover the sequence above is natural in morphisms $f$ satisfying the above hypotheses.
\end{lem}

\begin{proof}
If $\mathscr{C}$ is the homotopy cofiber of $f$, then there is a cofiber sequence
\[
\mathscr{X} \longrightarrow \mathscr{Y} \longrightarrow \mathscr{C} \longrightarrow \Sigma^1_s \mathscr{X} \longrightarrow \cdots.
\]
By assumption $\mathscr{F}$ is $\aone$-$(n-1)$-connected for some $n \geq 2$.  By the $\aone$-Freudenthal suspension theorem, $\Sigma^1_s \mathscr{F}$ is at least $\aone$-$n$-connected.  By the relative Hurewicz theorem \ref{thm:relativehurewicz}, we know that $\Sigma^1_s \mathscr{F} \to \mathscr{C}$ is an $\aone$-$(n+1)$-equivalence, so $\mathscr{C}$ is at least $\aone$-$n$-connected.

Since $\mathscr{C}$ is $\aone$-$n$-connected, it follows that $H^i(\mathscr{C},\mathbf{A}) = \hom(\bpi_i^{\aone}(\mathscr{C}),\mathbf{A})$ for $i \leq n$ by, e.g., \cite[Theorem 3.30]{ADExcision}.  The first statement then follows from the long exact sequence in cohomology associated with the above cofiber sequence.  For the second statement, observe that there are canonical isomorphisms
\[
\bpi_n^{\aone}(\mathscr{F}) \isomto \bpi_{n+1}^{\aone}(\Sigma^1_s \mathscr{F}) \isomto \bpi_{n+1}^{\aone}(\mathscr{C})
\]
by Morel's $\aone$-Freudenthal suspension theorem \cite[Theorem 6.61]{MField} and the relative Hurewicz theorem \ref{thm:relativehurewicz}.  Another application of \cite[Theorem 3.30]{ADExcision} then implies that
\[
H^{n+1}(\mathscr{C},\mathbf{A}) \cong \hom(\bpi_{n+1}^{\aone}(\mathscr{C}),\mathbf{A}),
\]
which then yields the identification
\[
H^{n+1}(\mathscr{C},\mathbf{A}) \cong \hom(\bpi_{n}^{\aone}(\mathscr{F}),\mathbf{A})
\]
by the isomorphisms stated above.  The functoriality statement is a consequence of the functoriality of the various construction involved.
\end{proof}

\begin{ex}
\label{ex:kinvariant}
In the notation of Lemma \ref{lem:describingkinvariants}, since $\mathscr{F}$ is $\aone$-$(n-1)$-connected, it follows that if $\mathbf{A}$ is any strictly $\aone$-invariant sheaf, then $H^n(\mathscr{F},\mathbf{A}) = \hom(\bpi_n^{\aone}(\mathscr{F}),\mathbf{A})$, again by \cite[Theorem 3.30]{ADExcision}.  In particular, taking $\mathbf{A} = \bpi_n^{\aone}(\mathscr{F})$, the identity morphism on $\bpi_n^{\aone}(\mathscr{F})$ gives a canonical element $1_{\mathscr{F}} \in H^n(\mathscr{F},\bpi_n^{\aone}(\mathscr{F}))$ that we refer to as the ``fundamental class of the $\aone$-homotopy fiber".  Then, consider the composite $d$ defined as:
\[
\xymatrix{
d:H^n(\mathscr{F},\bpi_n^{\aone}(\mathscr{F})) &  H^{n+1}(\mathscr{C},\bpi_n^{\aone}(\mathscr{F})) \ar[l]^-{\sim}\ar[r]^{\partial} & H^{n+1}(\mathscr{Y},\bpi_n^{\aone}(\mathscr{F})), \\
 & &
}
\]
where the left map is the canonical isomorphism discussed above (and arising from the relative Hurewicz theorem).  The class $d(1_{\mathscr{F}}) \in H^{n+1}(\mathscr{Y},\bpi_n^{\aone}(\mathscr{F}))$ that we refer to as the {\em transgression of the fundamental class of the $\aone$-homotopy fiber}.  The hypotheses of the previous result apply to the Moore-Postnikov factorization of a morphism of $\aone$-$1$-connected spaces.  In that case the element $d(1_{\mathscr{F}})$ is, by definition, the $k$-invariant at the relevant stage of the tower \cite[VI.5.5-6]{GoerssJardine}.
\end{ex}

\subsection{The obstruction theoretic Euler class is transgressive}
\label{ss:obstructiontheoreticeulerclass}
We now specialize the results of the previous subsection to the case of interest to obtain our reinterpretation of the Euler class.  We refer the reader to Subsection \ref{ss:contractionsandactions} for some preliminaries used here.  Consider the universal vector bundle $\gamma_n: \mathscr{V}_n \to Gr_n$.  Write $\mathscr{V}_n^{\circ}$ for the complement of the zero section.

\begin{lem}
The space $\mathscr{V}_n^{\circ}$ is $\aone$-weakly equivalent to $Gr_{n-1}$ in such a way that the cofibration $\mathscr{V}_n^{\circ} \to \mathscr{V}_n$ coincides with the map $Gr_{n-1} \to Gr_n$.
\end{lem}

\begin{proof}
We give an outline of the proof; we leave the reader the task of filling in the details.  We establish this in two steps.  First, observe that if $BGL_n$ is the usual simplicial classifying space of Nisnevich locally trivial $GL_n$-torsors, then there is a simplicial fiber sequence of the form
\[
GL_n/GL_{n-1} \longrightarrow BGL_{n-1} \longrightarrow BGL_n,
\]
where the second morphism is precisely that induced by functoriality of the simplicial classifying space construction applied to the standard inclusion map $GL_{n-1} \to GL_n$ (see, e.g., \cite[\S 2]{AHWII}).  One model for this fiber sequence is as follows: take $EGL_n$ (i.e., the Cech simplicial object attached to $GL_n \to \Spec k$) equipped with its usual diagonal $GL_n$-action.  The standard inclusion $GL_{n-1} \to GL_n$ induces an action of $GL_{n-1}$ on $EGL_n$ and the space $BGL_{n-1}$ then admits a model of the form $EGL_n/GL_{n-1}$.  One checks that this space is also simplicially weakly equivalent to $EGL_{n} \times^{GL_n} GL_n/GL_{n-1}$ and the projection onto the first factor induces the map $BGL_{n-1} \to BGL_n$ with fiber $GL_n/GL_{n-1}$.

Now, we simply observe that up to $\aone$-weak equivalence, we may replace all the spaces by those listed in the statement using repeatedly the fact that Zariski local $\aone$-weak equivalences are $\aone$-weak equivalence \cite[Example 2.3]{MV}. The projection onto the first column map induces an $\aone$-weak equivalence $GL_n/GL_{n-1} \to {\mathbb A}^n \setminus 0$.  The projection map $EGL_n \times^{GL_n} GL_n/GL_{n-1} \to EGL_n \times^{GL_n} {\mathbb A}^n \setminus 0$ is thus also an $\aone$-weak equivalence.  The projection map ${\mathbb A}^n \setminus 0 \to \Spec k$ factors through the inclusion map ${\mathbb A}^n \setminus 0 \hookrightarrow {\mathbb A}^n \setminus 0 \longrightarrow \Spec k$ and thus the map $EGL_n \times^{GL_n} {\mathbb A}^n \setminus 0 \to BGL_{n}$ thus factors as
\[
EGL_n \times^{GL_n} {\mathbb A}^n \setminus 0 \longrightarrow EGL_n \times^{GL_n} {\mathbb A}^n \longrightarrow BGL_n.
\]
The second map is an $\aone$-weak equivalence, while the first map is a cofibration.  Since for any $n \geq 0$, $BGL_n$ is $\aone$-weakly equivalent to $Gr_n$ \cite[\S 4 Proposition 3.7]{MV}, we conclude that $EGL_n \times^{GL_n} {\mathbb A}^n \setminus 0$ is $\aone$-weakly equivalent to $Gr_{n-1}$.  On the other hand, if we replace $EGL_n$ by the $\aone$-contractible Stiefel variety $St_{n}$ the we conclude also that $EGL_n \times^{GL_n} {\mathbb A}^n \setminus 0$ is $\aone$-weakly equivalent to ${\mathscr V}_n^{\circ}$, while $EGL_n \times^{GL_n} {\mathbb A}^n$ is $\aone$-weakly equivalent to $\mathscr{V}_n$.  However, this is precisely what we wanted to show.
\end{proof}

The cofiber of the inclusion $\mathscr{V}_n^{\circ} \to \mathscr{V}_n$ is, by definition, $Th(\gamma_n)$.

As observed in \S 3.3, the space $Gr_n$ is $\aone$-weakly equivalent to $BGL_n$, which is not $\aone$-$1$-connected.  On the other hand, we identified the $\aone$-universal cover of $BGL_n$ with a certain $\gm{}$-torsor over $BGL_n$.  Taking the model $Gr_n$ for $BGL_n$, the $\aone$-universal cover can be identified as in Remark \ref{rem:slcase}: it is the total space of complement of the zero section of the dual of the determinant of the tautological vector bundle over $Gr_n$; we write $\widetilde{Gr}_n$ for this model of the $\aone$-universal cover.  As before, we abuse notation and write $\gamma_n: \mathscr{V}_n \to \widetilde{Gr}_n$ for the universal bundle over $\widetilde{Gr}_n$, which comes equipped with a prescribed trivialization of the determinant.

There is an $\aone$-fiber sequence of the form
\[
{\mathbb A}^n \setminus 0 \longrightarrow BSL_{n-1} \longrightarrow BSL_{n}.
\]
We take as model for this fiber sequence the sequence
\[
{\mathbb A}^n \setminus 0 \longrightarrow \mathscr{V}^{\circ}_n \longrightarrow \mathscr{V}_n.
\]
Note that the homotopy cofiber of the map $\widetilde{Gr}_{n-1} \to \widetilde{Gr}_n$ is, by means of the above identifications, $Th(\gamma_n)$.

Applying Lemma \ref{lem:describingkinvariants} in this situation, gives a canonical class $o_n \in H^n(\widetilde{Gr}_n,\K^{MW}_n)$ as the transgression of the fundamental class in $H^{n-1}({\mathbb A}^n \setminus 0,\K^{MW}_n)$.  The pullback of $o_n$ along an $\aone$-homotopy class of maps $X \to BSL_n$ representing an oriented vector bundle yields the obstruction class $e_{ob}(\xi)$ for oriented vector bundles.  Note that, if $\xi: \mathcal{E} \to X$ is an {\em oriented} vector bundle, then this Euler class coincides with the (twisted) Euler class constructed before by functoriality with respect to pullbacks.

\begin{rem}
The case $n = 2$ is slightly anomalous.  In that case, note that, since $\bpi_2^{\aone}(BSL_2) = \K^{MW}_2$, we have $BSL_2^{(2)} = K(\K^{MW}_2,2)$.  The composite map $BSL_2 \to BSL_2^{(2)} \isomt K(\K^{MW}_2,2)$ defines the universal obstruction class in this case.  The reason for this discrepancy is that $BSL_1 = \ast$.  We have the model $\mathrm{HP}^{\infty}$ as a model for $BSL_2$ by the results of \cite{PaninWalterPontryaginClasses}.  The inclusion $\mathrm{HP}^1 \hookrightarrow \mathrm{HP}^{\infty}$ gives, up to $\aone$-homotopy, a map ${{\mathbb P}^1}^{\sma 2} \hookrightarrow BSL_2$ that factors the map ${{\mathbb P}^1}^{\sma 2} \to Th(\gamma_2)$.
\end{rem}


\subsection{Proof of Theorem \ref{thmintro:comparison}}
\label{ss:proofofmaintheorem}
We begin by studying a diagram that collects all the identifications we have made in the preceding sections.

\begin{prop}
\label{prop:diagramcommutes}
The following diagram commutes.
\[
\label{eqn:maindiagram}
\xymatrix{
H^{n-1}({\mathbb A}^n \setminus 0,\K^{MW}_n) \ar[dr]^{d} & & \\
H^n(Th(\gamma_n),\K^{MW}_n) \ar[u]^{\sim}\ar[r]^{\partial}\ar[d]^{\sim}& H^n(\mathscr{V}_n,\K^{MW}_n) \ar[d]^{\sim} & H^n(\widetilde{Gr}_n,\K^{MW}_n)\ar[l]_{\gamma_n^*} \ar[d]^{\sim}\\
H^n_{\widetilde{Gr}_n}(\mathscr{V}_n,G) \ar[r]^-{\partial} & \widetilde{CH}^n(\mathscr{V}_n)   & \widetilde{CH}^n(\widetilde{Gr}_n)\ar[l]_-{\gamma_n^*} \\
 \widetilde{CH}^0(\widetilde{Gr}_n) \ar[ur]_{(s_0)_*}\ar[u]^{\sim} & &
}
\]
where the upward pointing arrow at the bottom of the diagram is the d\'evissage isomorphism, and the upward pointing arrow in the first row comes from the relative Hurewicz theorem.
\end{prop}

\begin{proof}
The triangle on the bottom commutes by the discussion just prior to Definition \ref{defn:thomclass}.  The triangle on the top commutes by construction; see Example \ref{ex:kinvariant}.  The square on the right commutes because the identifications of Proposition \ref{prop:gerstencomplexescoincide} and Theorem \ref{thm:identification1} are functorial in $X$ by construction.  Finally, the commutativity of the left hand square is an immediate consequence of the definition of the group $H^n(Th(\gamma_n),\K^{MW}_n)$ by Corollary \ref{cor:thomspacesupports} (recall Notation \ref{notation:thomspaces}).
\end{proof}

Finally, we can establish Theorem \ref{thmintro:comparison} from the introduction.

\begin{thm}
\label{thm:comparison}
Under the canonical isomorphism $H^n(\widetilde{Gr}_n,\K^{MW}_n) \cong \widetilde{CH}^n(\widetilde{Gr}_n)$ of \textup{Theorem \ref{thm:identification1}}, the class $e_{ob}(\gamma_n)$ coincides with the class $e_{cw}(\gamma_n)$, up to multiplication by a unit in $GW(k)$.
\end{thm}

\begin{proof}
The diagram in Proposition \ref{prop:diagramcommutes} commutes and in this diagram all the morphisms in the left column are isomorphisms.  Under the Thom isomorphism, the group $H^n_{\widetilde{Gr}_n}(\mathscr{V}_n,G)$ or, equivalently, the group $H^n(Th(\gamma_n),\K^{MW}_n)$ is a free $\K^{MW}_0(k)$-modules of rank $1$ generated by the Thom class $t_{\gamma_n}$.

Now, take a rational point $x \in \widetilde{Gr}_n$.  By Proposition \ref{prop:functorialityofthomclasses}, the Thom class $t_{\gamma_n}$ has the property that $x^*t_{\gamma_n} = t_{\gamma_n|_x}$.  If we fix a trivialization of $\gamma_n|_x$, then there is an induced identification $Th(\gamma_n|_x) \cong {\pone}^{\sma n}$.  Thus, upon choice of a trivialization of $\gamma_n|_x$, restriction to the rational point $x$ yields a map
\[
H^n(Th(\gamma_n),\K^{MW}_n) \longrightarrow H^n({{\mathbb P}^1}^{\sma n},\K^{MW}_n).
\]
A priori, the choice of trivialization affects the identification of $H^n({{\mathbb P}^1}^{\sma n},\K^{MW}_n)$ as a $\K^{MW}_0(k)$-module; indeed it is free of rank $1$ generated by the Thom class, but there is a psychologically preferred choice.

Indeed, the bundle $\gamma_{n}$ over $\widetilde{Gr}_n$ is oriented, i.e., it comes with a preferred trivialization of its determinant (see Remark \ref{rem:slcase}).  For any point $x \in \widetilde{Gr}_n$, the pullback of this orientation yields a preferred orientation of $Th(\gamma_n|_x)$.  In particular, there is an induced isomorphism
\[
H^n(Th(\gamma_n),\K^{MW}_n) \longrightarrow H^n(Th(\gamma_n\vert_x),\K^{MW}_n)
\]
Fix a trivialization of $\gamma_n|_x$ that respects this orientation, and consider the induced map
\[
H^n(Th(\gamma_n),\K^{MW}_n) \longrightarrow H^n({\pone}^{\sma n},\K^{MW}_n).
\]


Next, observe that the isomorphism $H^n(Th(\gamma_n),\K^{MW}_n) \to H^{n-1}({\mathbb A}^n\setminus 0,\K^{MW}_n)$ is, by construction induced by the restriction to a fiber map $H^n(Th(\gamma_n),\K^{MW}_n) \to H^n(Th(\gamma_n|_x),\K^{MW}_n)$ an identification $H^n(Th(\gamma_n|_x),\K^{MW}_n) \cong H^n({\pone}^{\sma n},\K^{MW}_n)$ followed by the inverse of the suspension isomorphism $H^{n-1}({\mathbb A}^n \setminus 0,\K^{MW}_n) \isomt H^n({\pone}^{\sma n},\K^{MW}_n)$.  However, Lemma \ref{lem:sheafythomclasses} guarantees that under this identification the Thom class is sent to $\langle 1 \rangle \in \K^{MW}_0(k)$, i.e., the class of $1 \in \hom(\K^{MW}_n,\K^{MW}_n)$ under the identification $\hom(\K^{MW}_n,\K^{MW}_n) \cong H^n({\pone}^{\sma n},\K^{MW}_n)$.  Since these two classes are bases of a free rank $1$ $GW(k)$-module, they necessarily differ by a unit.
\end{proof}

\begin{rem}
With more work, we expect it is possible to establish the comparison result in the introduction for vector bundles that are not necessarily oriented, i.e., to check that the Euler classes twisted by the dual of the determinant coincide. Describing the $k$-invariant explicitly as the ``transgression" of a fundamental class in this setting is significantly more involved because one has to keep track of $\gm{}$-equivariance; the corresponding result in the setting of simplicial sets is \cite[Lemma VI.5.4]{GoerssJardine}.  We have avoided pursuing this generalization because in all cases we know where one wants to actually compute a twisted Euler class one uses the Chow-Witt definition.
\end{rem}

\subsection{Euler classes in Chow-Witt and Chow theory}
\label{ss:eulercherncompare}
If $X$ is a smooth scheme, then the Gersten complex defining Chow-Witt groups is constructed as a fiber product where one of the terms is the Gersten complex of Milnor K-theory.  In particular, for any line bundle ${\mathcal L}$ on $X$, there is morphism of complexes $C_r(X,G,{\mathcal L}) \to C_r(X,\K^M_r)$.  Taking cohomology, there are induced maps
\[
\widetilde{CH}^r(X,{\mathcal L}) \longrightarrow CH^r(X)
\]
that are functorial with respect to pullbacks.  We now study the image of the Euler class under such a map and thus provide a refinement of \cite[Proposition 6.3.1]{AsokFaselA3minus0}.

\begin{prop}
\label{thm:compatibilitywithchernclasses}
If $X$ is a smooth $k$-scheme, and $\xi: \mathcal{E} \to X$ is a rank $r$ vector bundle over $X$, then under the canonical map $\widetilde{CH}^r(X,\det \xi^{\vee}) \to CH^r(X)$, the class $e_{cw}(\xi)$ is mapped to $c_r(\xi)$.
\end{prop}

\begin{proof}
First, observe that $c_r(\xi)$ can be identified with $(\xi^*)^{-1} (s_0)_* 1 \in CH^r(X)$.  Indeed, since $\xi s_0 = id_X$, it follows that $(\xi s_0)^* = id$ and thus $(\xi^*)^{-1} = (s_0)^*$.  Therefore, \cite[Corollary 6.3]{Fulton} implies that $(\xi^*)^{-1} (s_0)_* 1 = (s_0)^* (s_0)_* 1 = c_r(\xi) \cap 1 = c_r(\xi)$.

Finally, naturality of the homomorphism from Chow-Witt groups to Chow groups guarantees that $\langle 1 \rangle \in \widetilde{CH}^0(X)$ is sent to $1 \in CH^0(X)$ (cf. \cite[Proposition 6.12]{FaselChowWitt}) and combined with the discussion of the previous paragraph yields the result.
\end{proof}

\begin{footnotesize}
\bibliographystyle{alpha}
\bibliography{comparingEulerclasses}
\end{footnotesize}
\Addresses
\end{document}